\documentclass[11pt,reqno]{amsart}
\usepackage{paralist}
\usepackage{enumerate}

\usepackage{a4wide,fullpage}
\usepackage{amssymb,amsmath,amsthm,mathrsfs,bm,bbm}
\usepackage{comment,enumitem,url}
\usepackage{xcolor}
\usepackage{cleveref}
\usepackage{mathtools}
\usepackage{bbm}

\newtheorem{theorem}{Theorem}
\newtheorem{lemma}[theorem]{Lemma}

\newtheorem{proposition}[theorem]{Proposition}
\newtheorem{definition}[theorem]{Definition}

\theoremstyle{remark}

\theoremstyle{definition}
\newtheorem{remark}{Remark}

\numberwithin{theorem}{section} \numberwithin{equation}{section}

\setlist[enumerate]{leftmargin=*,label=\rm{(\arabic*)}}
\setlist[itemize]{leftmargin=*}

\newcommand{\B}{\mathcal{B}}
\newcommand{\D}{\mathcal{D}}

\newcommand{\HE}{\mathcal{H}_\mathcal{E}}
\newcommand{\Hil}{\mathcal{H}}
\newcommand{\R}{\mathbb{R}}
\newcommand{\X}{\mathfrak{X}}

\newcommand{\E}{\mathcal{E}}

\usepackage{lipsum}

\newcommand\blfootnote[1]{%
  \begingroup
  \renewcommand\thefootnote{}\footnote{#1}%
  \addtocounter{footnote}{-1}%
  \endgroup
}

\def\lp{\left(}
\def\rp{\right)}

  \newenvironment{dedication}
  {
   \itshape             
   \raggedleft          
  }

\allowdisplaybreaks

\begin{document}

\title{An explicit construction of heat kernels and Green's functions in measure spaces}
\author{Palle Jorgensen \and Jay Jorgenson
\and Lejla Smajlovi\'{c}}

\date{\today}
\begin{abstract}
We explicitly construct a heat kernel as a Neumann series for certain function spaces, such as $L^{1}$, $L^{2}$, and Hilbert spaces,
associated to a locally compact Hausdorff space $\X$ with Borel $\sigma$-algebra $\mathcal{B}$, and
endowed with additional measure-theoretic data.
Our approach is an adaptation of classical work due to
Minakshishundaram and Pleijel, and it requires as input a parametrix or small time approximation to the heat kernel.
The methodology developed in this article applies to yield new instances of heat
kernel constructions, including normalized Laplacians on finite and infinite graphs as well as Hilbert spaces with reproducing
kernels.
\end{abstract}

\blfootnote{AMS Classification numbers.  35K08,  58J35, 44A35, 31C20, 46E30, 46E22, 47A10}
\maketitle

\begin{dedication}
\hspace{4cm}
{Dedicated to the memory of Professor Robert Strichartz.}
\end{dedication}

\section{Introduction}

\subsection{Imagining the heat kernel in the physical world}
From its inception, a certain class of heat kernels $K(x,y;t)$ consists of mathematical objects which are rooted in the physical world.   For example,
one can imagine introducing a unit burst of heat at a point $x$ in a domain $M$ at time zero, after which one then can measure
the proportion of heat which has diffused to the point $y$ after time $t$.  This amount of heat is denoted by $K(x,y;t)$.
Equivalently, one can envision a particle which undertakes a random walk on $M$ beginning at $x$ at time zero, so then
$K(x,y;t)$ is interpreted as the probability that the particle is at the point $y$ at time $t$.

Without exaggeration, there
are hundreds of mathematics articles which study the heat kernel in various settings for $M$, such as Euclidean domains,
Riemannian manifolds, projective K\"ahler varieties, symmetric spaces, and so on, and with some articles considering different possibilities
for time, including positive real time, positive integer time, or arbitrary time-scales; see \cite{BP01}.   In our opinion, it is valid to
keep the physical or probabilistic language when describing properties of a heat kernel $K(x,y;t)$ in the case when $M$
is modeled locally on a Euclidean domain since this means that in one way or another one can draw $M$, hence one can imagine diffusion or a random walk.

\subsection{A general point of view}
Beyond a physical or probabilistic setting, there is an extensive body of literature which considers heat kernels more generally,
such as when $M$ is replaced by  a locally compact Hausdorff space $\mathfrak{X}$ with a Borel $\sigma$-algebra  $\mathcal{B}$  of subsets
and possessing a $\sigma$-finite measure $\lambda$ on $(\X,\mathcal{B})$.
We refer the reader to the particularly eloquent article
\cite{Gr03} where the author defines and studies a heat kernel in a metric measure space;  see, specifically, Definition 2.1 on page 146 of \cite{Gr03}.
With various assumptions, it is shown in \cite{Gr03} that a heat kernel can give rise to a Markov process for which the heat kernel is
the associated transition density; see  loc. cit., page 147.
  In this regard, the heat kernel retains its connection to probability theory.
We also note that the heat kernel in these instances has an associated
semigroup of operators, and a heat operator is obtained by considering the infinitesimal generator of
this semigroup.   In essence, the infinitesimal generator is a type of Laplacian operator.

Our main goal in the present article is to explicitly construct a heat kernel as an integral kernel
in certain measure space settings, similar to those considered in \cite{Gr03}, with positive real time;
see Theorems \ref{thm: formula for heat kernel}, \ref{thm: formula for heat kernel 2} and
 \ref{thm: formula for H- heat kernel}.
  In effect,
we assume that, somehow, we have obtained a parametrix or small time approximation for the heat kernel
as an integral kernel.  Given this approximation, the heat kernel
itself is computed using a convergent Neumann series which is defined using the parametrix.  As is commonplace, the
parametrix itself has reasonably few requirements, so in many instances our construction yields an easily attainable result.

The parametrix approach we present is modeled on the classical method from the
Riemannian setting
as given the seminal paper \cite{MP49}, which itself attributes some of the main ideas to Carleman; see loc. cit., page 243.
In the recent articles \cite{CJKS24} and \cite{JKS26} the authors developed the parametrix approach for constructing heat kernels in
the setting of finite and certain infinite graphs, respectively.  Again, there is considerable flexibility in choosing a parametrix.
For example, it is shown in \cite{CJKS24} and \cite{JKS26} that different choices of a parametrix lead to different identities, including
expressions for the classical $I$-Bessel function.  With that in mind, one can say that the parametrix construction of the
heat kernel has been shown to be quite robust.

There are numerous studies in the literature which study heat kernels in metric measure spaces; see, for example \cite{Gr03} and
references therein.  As far as we can see, the present article is the first which explores an explicit construction of a heat
kernel starting with a parametrix in such settings.

\subsection{Related studies}
Our study is motivated with a view toward analysis on infinite graphs in a general framework, namely:
(i) Generalized integral kernels and transition kernels, with an emphasis on generalized graph Laplacians, their induced Green's functions and
potential theoretic properties;  (ii) the relevant potential theoretic constructions which we shall make use of; and (iii) families of associated
stochastic processes, directly related to (i) and (ii). Our discussion below of these themes will be supplemented by background references.
While the general literature is vast, we, of course, will focus on the details that are tailored to our needs.  These ideas are connected to recent papers
by P. Jorgensen (first named author) and his (other) co-authors, especially the following:
With S. Bezuglyi for (i) (see \cite{BJ23} and references therein), 
with E. Pearse for (ii) (see \cite{JP19} and references therein) as well as \cite{JRT25}, 
and with F. Tian for (iii) (see \cite{JT25} and references therein). 


While there is a rich variety of frameworks which involve incarnations of Green's functions and associated diffusions, each case and each application entails unique choices of both function spaces and associated Laplacians.  
An interesting new development which has served as motivation for our present results is an adaptation of graph Laplacians to a continuous Borel framework. In the various settings one then arrives at a natural choice of a measure $\mu$ and an associated $L^p$ space where the Laplacian has dense domain. This is the framework we follow.
We are then able to obtain global results where the time variable enters at the outset and, therefore, yields a more direct solution to the particular diffusion equations for the setting at hand.  We refer to the following earlier papers which deal with a variety of settings of generalized Laplacians, heat kernels, Green's functions, boundary spaces, and potential theory:   \cite{AJ24, B-GS-KSY06, BJ24, JS23, JST24, JT25b, St90}.  This list of citations is not complete, but there are additional papers cited in these articles for further reading.

\subsection{Organization of the paper.}
In section 2 we define the setup for our study and set notation.  In section 3
we define the heat kernel in our context.  Though this component of our paper is similar to considerations by other authors, we want to carefully state
the minimal number of conditions needed to proceed.  In section 4 we define the convolution to be employed and prove its basic properties
in the setting of $L^{p}$ spaces; section 5 consists of an analogous study for separable Hilbert spaces.  It is important to note that
in each instance the notation of convergence means, in general terms, with respect to the norm of the given space.  In section 6 we define
a parametrix in each context we study, and prove the existence of a heat kernel through explicit construction in terms of a convergent Neumann series.
{\it Thus, existence of the heat kernel in each context follows from existence of the corresponding parametrix.}
In section 7 we discuss the important question as to how one can construct a parametrix itself and set the stage for further studies of this question in specific settings.  In section 8 we show how other properties of the heat kernel follow from our analysis,
and perhaps additional assumptions which, again, need to be viewed on a case-by-case basis.
In section 9 we describe how one can proceed from the heat kernel to other important
considerations, such as Green's
function, resistance measures, wave kernels and, in some cases, a pre-trace formula.  We finish with section 10 where we discuss examples of parametrix for graphs and reproducing kernel Hilbert spaces and offer some concluding remarks.

\medskip
\section{Our setup}\label{sec. setup}

Let $\mathfrak{X}$ denote a locally compact Hausdorff space, and $\mathcal{B}$ a Borel $\sigma$-algebra of subsets of $\X$.
Let $\lambda$ denote a $\sigma$-finite measure on $(\X,\mathcal{B})$. Let
$$
\mathcal{B}_{\rm fin}:=\{A\in \B: \lambda (A)<\infty\}
$$
denote the subalgebra of sets of finite $\lambda$ measure with associated function spaces
$$
\D_{\rm fin}=\text{span}\{\chi_A: A\in \mathcal{B}_{\rm fin}\}
\,\,\,\,\,
\text{\rm and}
\,\,\,\,\,
\D_{\rm fin}^{\dagger}=\text{span}\{\chi_A: A\in \mathcal{B}_{\rm fin}\cup \{\X\}\},
$$
where $\chi_{A}$ is the indicator function of $A \in \mathcal{B}$.
If $\lambda(\X)<\infty$, then $\D_{\rm fin}=\D_{\rm fin}^{\dagger}$, otherwise we have
that $\D_{\rm fin}=\D_{\rm fin}^{\dagger} \cup \{\chi_{\X}\}$.
We note that any $\sigma$-finite measure $\lambda$ is uniquely determined by its values on $\mathcal{B}_{\rm fin}$.

\medskip
\subsection{Transfer operators and conductance}
As described in \cite{JP19}, one defines a transfer operator $R$ on
$(\X,\mathcal{B},\lambda)$ as an operator on $\D_{\rm fin}^{\dagger}$ which satisfies
the following two conditions.

\begin{enumerate}[leftmargin=0.75in]
  \item  The operator $R$ is symmetric, meaning that
  $$
  \int\limits_{\X}(Rf)(x)g(x)\lambda(dx)=\int\limits_{\X}f(x)(Rg)(x)\lambda(dx)
  \,\,\,\,\,
  \text{\rm for all}
  \,\,\,\,\,
  f,g\in \D_{\rm fin}^{\dagger}.
  $$
  \item The operator $R$ is semi-positive, which we also call non-negative, meaning that for all $f\in \D_{\rm fin}^{\dagger}$ we have
that  $Rf\geq0$ whenever  $f\geq 0$.
\end{enumerate}

\medskip
\noindent
Associated to any transfer operator $R$, there exists a measure $\rho$
 on $(\X^2,\mathcal{B}\otimes \mathcal{B})$ such that $\lambda=\rho\circ \pi_i^{-1}$, where $\pi_i$ is projection of $\X^2$
  onto the variable $i$, where $i$ equals $1$ or $2$,  and for which $\rho_x:=\rho(x,\cdot)$ is the kernel for $R$.  In other words,
 $$
 (Rf)(x)=\int\limits_{\X}f(y)\rho_x(dy)
 $$
 for all $f\in \mathcal{D}_{\mathrm{fin}}$.
The measure $\rho$ is often called the \emph{conductance}, and it is a symmetric measure on $(\X^2, \mathcal{B}\otimes \mathcal{B})$, meaning that
$$
\rho(A\times B)=\rho(B\times A)
\,\,\,\,\,
\text{for all $A,B\in\mathcal{B}$.}
$$
Moreover, $\rho$ is a positive measure and
$$
0\leq \rho_x(\X)<\infty
\,\,\,\,\,
\text{\rm for all $x\in\X$.}
$$
Finally, we recall from \cite{BJ23} that a positive transfer operator $R$ is symmetric if and only if it defines a symmetric measure $\rho$ by
$$
\rho(A\times B)=\int\limits_{\X}\chi_A(x)R(\chi_B)(x)\lambda(dx).
$$

\begin{remark}\label{rem:graph}
If $X$ is a discrete space whose elements are viewed as vertices of a weighted graph, then the measure
$\rho$ are edge weights. In this instance, an example of a transfer operator $R$ is the operator determined by the
associated adjacency matrix.
\end{remark}

\begin{remark}\label{rem:measure_to_transfer}
A different approach to defining a transfer operator, stemming from the properties of the symmetric measures, is described in \cite{BJ23}.
To do so, start with a $\sigma$-finite symmetric measure on $(\X^2,\mathcal{B}\times\mathcal{B})$ called $\tilde{\rho}$ and set  $\lambda:=\tilde\rho\circ\pi_1^{-1}$. Then there exists a unique system of conditional $\sigma$-finite measures $\tilde \rho_x$
for each $x \in \X$ such that
  $$
  \iint\limits_{\X^2}f(x,y)\tilde \rho (dx,dy) = \int\limits_\X \left(\int\limits_\X f(x,y) \tilde\rho_x(dy)\right)\lambda(dx),
  $$
  for any $\tilde\rho$-integrable function $f$. With this set-up, the operator $\tilde R$  defined on $\mathcal{D}_{\rm fin}^{\dagger} $ by
  $$
  \tilde{R}(g)(x)=\int\limits_\X g(y)\tilde\rho_x(dy)
  $$
  is symmetric and positive.  As such, $\tilde{R}$ is a transfer operator on $(\X,\mathcal{B},\lambda)$.
\end{remark}

\medskip
\subsection{The Laplacian}

With the above notation, for a given transfer operator $R$ let us define a degree function $c(x)$ associated to $R$ by
 \begin{equation}\label{eq. c defn}
 c(x)=R( \chi_{\X})(x)=\int\limits_{\X}\rho_x(dy)=\rho_x(\X)<\infty.
 \end{equation}
 In line with the discussion of Remark \ref{rem:graph}, the degree function can be viewed as the analogue of
 the vertex degree of a finite graph.

Also, as discussed in Remark \ref{rem:measure_to_transfer}, one can define a degree function associated to a symmetric measure $\tilde \rho$ as the degree function associated to the
transfer operator $\tilde R$ generated by $\tilde\rho$.

Going forward, we
assume that the degree function $x\mapsto c(x)=\rho_x(\X)$  satisfies the following conditions.

\medskip
\textbf{Assumption C.} For $\lambda$-almost every $x \in \X$, we have that
$$
 0<c(x)<\infty.
$$
Additionally we assume, that $c\in L_1^{\rm loc}(\lambda)$;  specifically, it is assumed that
$$
\int\limits_A c(x)\lambda(dx)< \infty
\,\,\,\,\,
\text{\rm for all}
\,\,\,\,\,
A\in \mathcal{B}_{\rm fin}.
$$

\begin{definition}\label{def:Laplacian}
   The Laplacian on $(\X, \mathcal{B},\lambda)$ is the closure of the symmetric operator
   $$
   \Delta :=(c-R)
   $$
   on the domain $\D_{\rm fin}^{\dagger}$. In other words, the Laplacian $\Delta f$ of $f$ is
   defined as
 $$
 \Delta f(x)= \int\limits_{\X}(f(x)-f(y))\rho_x(dy)
\,\,\,\,\,
\text{\rm for any $f \in \D_{\rm fin}^{\dagger}$.}
$$
\end{definition}

\begin{remark}
Given that $c(x)>0$ for  $\lambda$-almost every $x\in\X$,
one can also define a $\lambda$-almost every analogue of the probabilistic/normalized Laplacian $\tilde \Delta$ by
$$
\tilde \Delta: = \frac{1}{c}\Delta = I - \frac{1}{c}{R},
$$
so then
$$
\tilde \Delta f(x)=\frac{1}{c(x)} \int\limits_{\X}(f(x)-f(y))\rho_x(dy).
$$
The normalized Laplacian can be written in terms of the Markov operator $P$
$$
P(x,dy):= \frac{1}{c(x)} \rho_x(dy),
$$
where the measure $c^{-1}(x) \rho_x(dy)$
is simply the probability measure obtained by normalizing of $\rho_x(dy)$ to have unit mass.   Explicitly, we can write that
$$
(Pf)(x):=\frac{1}{c(x)}(Rf)(x)=\frac{1}{c(x)}\int\limits_{\X}f(y)\rho_x(dy)=
\int\limits_{\X}f(y)P(x,dy).
$$
With this notation, we have that
$$
\tilde \Delta f=(I-P)f
$$
on the domain $\D_{\rm fin}^{\dagger}$.
\end{remark}

\medskip
\subsection{Hilbert spaces and measure spaces of interest}
Note that under Assumption C, one can define another measure $\nu$ on the measurable space $(\X,\mathcal{B})$ by
\begin{equation} \label{eq. nu defn}
\nu(A):=\int_A c(x)\lambda(dx)=\int_A \int_\X\rho_x(dy)\lambda(dx)=\rho(A\times \X).
\end{equation}
The measure spaces  $L^{2}(\lambda)$ as well as $L^{2}(\nu)$ will be studied going forward.
Additionally, we will consider so-called energy space $\HE$ which is defined as follows.

\begin{definition} For any $f, g \in \D_{\rm fin}^{\dagger}$, define the energy inner product $  \langle\cdot, \cdot \rangle_{\E}$ by
 $$
  \langle f,g\rangle_{\E}:=\frac{1}{2}\iint\limits_{\X^2}(f(x)-f(y))(g(x)-g(y))d\rho(x,y).
  $$
The associated norm is $||f||_{\E} =\sqrt{\langle f,f\rangle_{\E}}$.  With this,
the energy space $\HE$ is defined as
  $$
  \HE:=\{f\in \D_{\rm fin}^{\dagger}: ||f||_{\E}<\infty\}.
  $$
\end{definition}

\medskip
\noindent
We will use the generic notation $\Hil$ to denote a Hilbert space, including any one of
the Hilbert spaces $L^{2}(\lambda)$, $L^{2}(\nu)$ or $\HE$.

\medskip
Going forward, we will need the following lemma on differentiating an inner product in a Hilbert space $\Hil$ with respect to a parameter.

\begin{lemma} \label{lem: dif of inner product wrt t}
   Let $f:\X\times(0,\infty)$ be a function such that: (i) for any $t\in(0,\infty)$, function $x\mapsto f(x,t)$ belongs to $\Hil$, and (ii) for any $x\in\X$, function $t\mapsto f(x,t)$ is differentiable. Further assume that
   $$
   \partial_t f(x,t)\in \Hil
   \,\,\,\,\,
   \text{\rm for all $x\in\X$.}
   $$
    Then
   $$\partial_t \langle f(x,t),g(x)\rangle_{\Hil} = \langle \partial _t f(x,t),g(x)\rangle_{\Hil}
   \,\,\,\,\,
   \text{\rm for all $g\in\Hil.$}
   $$
\end{lemma}
\begin{proof}
The statement follows by applying the Schwarz inequality for the inner product on $\Hil$.
\end{proof}

A portion of our discussion below will consider the space $L^{1}(\lambda)$ and
its subspace $L^{1}(\lambda) \cap L^{\infty}(\lambda)$.  In this setting, and more generally,
it is immediate, by applying the classical H\"older inequality, that one can state and prove an $L^{p}$ version of Lemma \ref{lem: dif of inner product wrt t}, which we state below.

\begin{lemma}\label{rem:Lp_differentiation}
Let $p, q \geq 1$ be real numbers such that $1/p+1/q = 1$, with the usual convention that
$q=\infty$ if $p=1$ and vice versa. Assume $f:\X\times(0,\infty)$ is a function with the following properties: (i) for any $t\in(0,\infty)$, function $x\mapsto f(x,t)$ belongs to $L^{p}(\lambda)$, and (ii) for any $x\in\X$, function $t\mapsto f(x,t)$ is differentiable. Further assume that
   $$
   \partial_t f(x,t)\in L^{p}(\lambda)
   \,\,\,\,\,
   \text{\rm for all $x\in\X$.}
   $$
    Then
   $$\partial_t \int\limits_\X f(x,t)g(x) \lambda(dx) = \int\limits_\X \partial _t f(x,t)g(x)\lambda(dx)
   \,\,\,\,\,
   \text{\rm for all $g\in L^{q}(\lambda).$}
   $$
\end{lemma}

\medskip
\subsection{Properties of the Laplacians }
In general terms, the Laplacians we study are semi-positive and (essentially) self-adjoint.  Let us discuss briefly
where these results are shown to be true in the various spaces we consider.
 Throughout our paper, following \cite{BJ19a}, we will pose an
additional assumption on the symmetric measure $\rho$ on  $(\X^2,\mathcal{B}\otimes \mathcal{B})$, which is the following.
\medskip

\textbf{Assumption E. } For every set $A\in\mathcal{B}_{\rm fin}$, the function
$$
x\mapsto \rho_x(A)=\int\limits_{\X}\chi_A(y)\rho_x(dy)
$$
belongs to $L^1(\lambda)\cap L^2(\lambda)$.

\medskip
\noindent
Under this assumption, it is proved in Sections 7 and 8 of \cite{BJ19a} that the Laplacian $\Delta$ introduced in Definition \ref{def:Laplacian} is a positive definite
and essentially self-adjoint in the
Hilbert space $L^2(\lambda)$. Moreover, $\Delta$ acting in $L^2(\lambda)$ is bounded if and only if
$c\in L^{\infty}(\lambda)$, where function $c$ is defined by \eqref{eq. c defn}.
According to \cite[Lemma 3.5]{JP19}, the Markov operator $P$ is self-adjoint on $L^2(\nu)$, hence so is the normalized Laplacian $\tilde \Delta=I-P$.

When viewed as an operator on the energy space $\HE$, the operator $\Delta$ is a positive definite and symmetric operator, but in general it is not
 self-adjoint.  However, $\Delta$ does admit a self-adjoint extension; see Theorem 8.5 of \cite{BJ19a}.  We refer an interested reader to
 \cite[Theorem 2.14]{BJ23} where further properties of operators $R$ and $P$ are summarized.  Moreover, the Laplacian, when viewed as an operator
 in the energy Hilbert space $\HE$  is semi-bounded on its natural dense domain; hence, by Friedrichs’ theorem, $\Delta$ has a canonical semi-bounded
 extension with the same lower bound as $\Delta$ on $\D_{\rm fin}^{\dagger}$.

The Laplacians in the $L^2$ spaces $L^2(\lambda)$ and $L^2(\nu)$ (where the measure $\nu$ is defined by \eqref{eq. nu defn}) are automatically essentially self-adjoint; for details,
see, for example, \cite{JP17, JP19, JPT18}.

\medskip
\section{Heat kernel on $L^{p}$ spaces and Hilbert spaces}

We will now define a heat kernel in the various spaces under consideration.  Further properties of
the heat kernel will be established in a later section with further assumptions as needed.

\begin{definition} \label{def: HK lp space} Let $\mu$ be a measure on the space $(\X,\mathcal{B})$.
  Let $p\geq 1$ be a real number and let $q$ be such that  $\frac{1}{p}+\frac1q=1$, where as usual for $p=1$ we set $q=\infty$ and vice versa.
  The $L^p$-heat kernel associated to the Laplacian $\Delta$ is a function $K:\X\times\X\times[0,\infty)\to\mathbb{R}$ with the following properties.
  \begin{enumerate}[leftmargin=0.75in]
    \item For all $x,y\in\X$, the function $K(x,y;\cdot):[0,\infty)\to\mathbb{R}$ is differentiable in $t \in (0,\infty)$, continuous at $t=0$
     and such that $K(\cdot, y;t):\X\to\mathbb{R}\in L^p(\mu)$ for all $y\in\X$ and $t>0$.
    \item The function $K(\cdot,y;\cdot):\X\times[0,\infty)\to\mathbb{R}$ is a solution to the equation $(\partial_t+\Delta)K(x,y;t)=0$
for all $(x,t)\in \X\times(0,\infty)$, fixed $y\in\X$, $\mu$-almost everywhere, and where $\Delta$ is assumed to act on the variable $x$.
\item For all $f\in L^q(\mu)$, we have that
    \begin{equation}\label{eq. dirac delta cond HK}
\lim\limits_{t \rightarrow 0^{+}}\int\limits_{\X}K(x,y;t)f(y)\mu(dy)=f(x)
\,\,\,\,\,
\text{\rm $\mu$-almost everywhere $x\in\X$.}
\end{equation}
  \end{enumerate}
\end{definition}

Note that the Laplacian depends upon a measure $\lambda$ on the space $(\X,\mathcal{B})$,
which is not necessarily equal to the measure $\mu$ in the above definition.

In an analogous way it is possible to define the heat kernel on the Hilbert space $\mathcal{H}$.

\medskip
\begin{definition}\label{def: HK Hilb space}
  Let $\mathcal{H}$ be a Hilbert space of functions which has $\X$ as a domain. The heat kernel on $\mathcal{H}$ associated to the Laplacian $\Delta$ is a function $K:\X\times\X\times[0,\infty)\to\mathbb{R}$ with the following properties.
  \begin{enumerate}
    \item For all $x,y\in\X$, $K(x,y;\cdot):[0,\infty)\to\mathbb{R}$ is differentiable on $(0,\infty)$, continuous at $t=0$ and such that $K(\cdot, y;t)\in \mathcal{H}$ for all $y\in\X$, $t>0$.
    \item The function $K(\cdot,y;\cdot):\X\times[0,\infty)\to\mathbb{R}$ is a solution to the equation
$(\partial_t+\Delta)K(x,y;t)=0$ for all $t>0$, where $\Delta$ acts on $x$, for fixed $y\in\X$.
    \item For all $f\in \mathcal{H}$, we have that
    \begin{equation}\label{eq. dirac delta cond HK2}
\lim\limits_{t\rightarrow 0^+} \langle K(x,\,\cdot;t), f(\cdot)\rangle_{\mathcal{H}}=f(x), \quad x\in\X,
\end{equation}
where $\langle \cdot, \cdot \rangle_{\mathcal{H}}$ denotes the inner product on $\mathcal{H}$.
  \end{enumerate}
\end{definition}

\medskip
\begin{remark} If $\Hil$ is a reproducing kernel Hilbert space, then one can replace \eqref{eq. dirac delta cond HK2}
with the statement that
\item \begin{equation}\label{eq. dirac delta cond HK3}
\langle K(x,\,\cdot;0), f(\cdot)\rangle_{\mathcal{H}}=f(x)
\,\,\,\,\,\text{\rm for all $x\in\X$}
\end{equation}
and for all $f\in \mathcal{H}$, where $\langle \cdot, \cdot \rangle_{\mathcal{H}}$ denotes the inner product on $\mathcal{H}$.
In other words, $K(x,y;0)$ is the reproducing kernel for the Hilbert space $\mathcal{H}$.
\end{remark}


\begin{remark}
In \cite{Gr03} the author defines a heat kernel on a measure space which includes a number of additional
properties which are assumed to hold; see also \cite{GHL14}.  The additional properties assumed in \cite{Gr03} and \cite{GHL14}, and
elsewhere, include well-known features from geometry and probability such as positivity and symmetric of the heat kernel as
well as a semi-group property.  The development we give in this paper is to use the above assumptions to prove the
existence of the heat kernel through an explicit means, assuming one has a parametrix,
after which we examine the circumstances under which the additional
assumptions from \cite{Gr03} and \cite{GHL14} become properties of the heat kernel.
\end{remark}

\begin{remark}
The definition of the heat kernel is an adaptation from the setting of manifolds to the situation of $L^p$ or Hilbert spaces;
however, there is an important technical difference.
Namely, if $M$ is a manifold with associated volume element $\omega$, then the heat kernel $K_M$ is assumed
to be continuous and integrable in the two space variables, smooth in the time variable on $(0,\infty)$ and such that
$$
\lim_{t\to0^+ } \int\limits_M K_M(x,y;t)f(y) \omega(dy)=f(x),
$$
for all bounded, continuous functions on $M$. In the setting of a Hilbert space, not all bounded and continuous
functions belong to this space; the same is true for the space $L^q(\mu)$.
\end{remark}

\begin{remark}
  The definition of the $L^p$-probabilistic heat kernel and the Hilbert space probabilistic heat kernel associated to the
  probabilistic/normalized Laplacian $\tilde\Delta$ is the same as above, with the only difference being  that one needs
  to take the Laplacian $\tilde\Delta$ instead of $\Delta$ in the second condition. In the sequel we will work mainly with
  $\Delta$. However, it is straightforward to see that all our findings will hold true with $\Delta$ replaced by $\tilde\Delta$
  and $L^2(\mu)$ replaced by $L^2(\nu)$.
\end{remark}

\section{Generalized time convolution}

In classical analysis, convolution of functions is used in different contexts. Usually, the convolution of two functions on some space is integral of a product of a function and its shift, or the integral of a product of a function with a certain kernel. It is very useful, because, loosely speaking, it can help  smoothen the function, meaning that convolution produces a new function on the underlying space which possesses somewhat nicer properties.

In our setup we are interested in constructing a heat kernel which is a fundamental solution for the heat equation and describes diffusion over time. In other words, the heat kernel, introduced in Definitions \ref{def: HK lp space} and \ref{def: HK Hilb space}, depends on two space variables and a time variable.
As such, the convolution of functions depending on two space variables and one time variable is defined as a composition of the "classical" convolution in the space variable, followed by a convolution in the time variable.

In order to construct the heat kernel,
we need the notion of somewhat more generalized time convolution, or a time-space convolution,
which is adopted to the setting of general measurable or Hilbert spaces.

\medskip


\subsection{Generalized time convolution on measurable spaces}

Recall that throughout this section $\mu$ is a measure on the space $\X,\mathcal{B}$.

\begin{definition} \textbf{Time convolution on} $L^p$ \textbf{space}. \label{def. conv}
Assume the following.
\begin{itemize}
\item The functions $F_1,F_2: \X\times\X\times[0,\infty)\to \mathbb{R}$ (or $\mathbb{C}$) are such
that $F_1(x,\cdot;t)\in L^p(\X,\mu)$ and $F_2(\cdot,y;t)\in L^q(\X,\mu)$ for all $x,y\in\X$, $t\geq 0$ and $p,q\in[1,\infty]$ with $1/p+1/q=1$.
\item The functions $F_1(x,y;\cdot)$ and $F_2(x,y;\cdot)$ are integrable on the sets $[0, b)$, for all $b>0$ and $x,y\in\X$.
\end{itemize}
\noindent
Then we define the convolution of $F_1$ and $F_2$ as
\begin{equation}\label{eq. defn conv F1,2}
(F_1\ast F_2)(x,y;t):=\int\limits_{0}^t\lp\int\limits_{\X}F_1(x,z;t-\tau)F_2(z,y;\tau)\mu(dz)\rp d\tau.
\end{equation}
\end{definition}

Trivially, by combining the Fubini-Tonelli theorem with the H\"older inequality we get that
\begin{align*}
(F_1\ast F_2)(x,y;t)&=\int\limits_{0}^t\lp\int\limits_{\X}F_1(x,z;\tau)F_2(z,y;t-\tau)\mu(dz)\rp d\tau
\\&= \int\limits_{\X}\lp\int\limits_{0}^t F_1(x,z;\tau)F_2(z,y;t-\tau) d\tau\rp \mu(dz).
\end{align*}

\begin{remark}
It is straightforward to see that \eqref{eq. defn conv F1,2} generalizes the definition of the convolution for
infinite, edge-weighted and vertex-weighted graphs as given on page 8 of \cite{JKS26}.
\end{remark}


Let us now deduce some properties of the generalized convolution which will be needed for the construction of the heat kernel.
For $p\in [1,\infty]$ we denote by $||\cdot||_{p,\mu}$ the norm in the space $L^p(\mu)=L^p(\X,\mathcal{B},\mu)$.

\begin{lemma}
\label{lem:convolution bounds} Let $F_{1},F_{2}:\X\times \X\times\mathbb{R}_{>0}\to\mathbb{R}$
be as in Definition \ref{def. conv}.  For some $t_{0} > 0$, assume there exist constants $C_1, C_2$ and integers $k,\ell\geq 0$
such that for all $0<t<t_{0}$ and all $x,y \in \X$, we have that
$$
\left\|F_{1}(x,\cdot;t)\right\|_{p,\mu}\leq C_{1}h(x)t^{k}
  \,\,\,\,\,
  \text{\rm and}
  \,\,\,\,\,
  \left\|F_{2}(\cdot, y;t)\right\|_{q,\mu}\leq C_{2}t^{\ell},
$$
for some function $h$  of $x\in\X$.
Then for all $x,y\in \X$
\[
  |(F_{1}\ast F_{2})(x,y;t)|\leq C_{1}C_{2}h(x)\frac{k!\ell!}{(k+\ell+1)!}
  t^{k+\ell+1}
  \,\,\,\,\,
  \text{\rm for $0<t<t_{0}$.}
\]
\end{lemma}

\begin{proof}
From the H\" older inequality, we can write that
\begin{align*}
|(F_{1}\ast F_{2})(x,y;t)|&\leq \int\limits _{0}^{t} \left\| F_1(x,\cdot;t-r)\right\|_{p,\mu} \left\| F_2(\cdot, y;r)\right\|_{q,\mu}dr\\
&\leq C_1C_2h(x)\int_0^t r^k (t-r)^\ell\, dr=C_1C_2h(x)\frac{k!\ell!t^{k+\ell+1}}{(k+\ell+1)!},
\end{align*}
as claimed.
\end{proof}

Let $f=f(x,y;t):\X\times \X\times\mathbb{R}_{>0}\to\mathbb{R}$ be a function with the following properties.  For any  $T>0$, all $t\in (0,T]$,
and all arbitrary but $\mu$ almost all fixed $x,y\in \X$, the functions $f(\cdot, y;t): \X \to \mathbb{R} $ and $f(x,\cdot;t): \X \to \mathbb{R} $
belong to $L^p(\mu)\cap L^q(\mu)$.  Also, assume that, with the same quantification, the function $f(x,y;\cdot) :(0,T] \to \mathbb{R}$ is integrable.
For any positive integer $\ell$ and any such function $f$, we can inductively define the $\ell$-fold
convolution $(f)^{\ast\ell}(x,y;t)$ for $t\in(0,T]$; set $(f)^{\ast1}(x,y;t)=f(x,y;t)$
and for $\ell\geq2$ let
$$
(f)^{\ast\ell}(x,y;t):=\left(f\ast(f)^{\ast(\ell-1)}\right)(x,y;t),
$$
under additional assumption that $(f)^{\ast(\ell-1)}(\cdot,x;t)\in L^q(\mu)$ for all $t\in(0,T]$ and all $\ell \geq 2$.

With this notation we have the following lemma.

\begin{lemma}
\label{lem:convergence of the series} Let $f=f(x,y;t):\X\times \X\times\mathbb{R}_{>0}\to\mathbb{R}$.
Assume that for all $\mu$ almost all $x,y\in \X$ and all $t_{0} \in \mathbb{R}_{\geq0}$, the function $f(x,y;\cdot)$ is integrable on
 $[0,t_{0}]$ and $f(x,\cdot;t) \in L^1(\mu)$ for all $t\in [0,t_0]$.
Assume further that for all $t_{0} \in \mathbb{R}_{>0}$ there exists a constant $C$, depending only upon $t_0$, and integer $k\geq 0$ such that
$$
|f(x,y;t)|\leq Ct^{k}
\,\,\,\,\,
\text{\rm for $\mu$ almost all $x,y\in \X$ and $0<t<t_{0}$.}
$$
Then the series
\begin{equation}\label{eq: series over ell}
F(x,y;t):=\sum_{\ell=1}^{\infty}(-1)^{\ell}(f)^{\ast\ell}(x,y;t)
\end{equation}
converges absolutely and uniformly on every compact subset of $\X\times \X\times\mathbb{R}_{\geq0}$.
In addition, we have the following results:
\begin{enumerate}
\item \begin{equation}
\left(f\ast\left(\sum_{\ell=1}^{\infty}(-1)^{\ell}(f)^{\ast\ell}\right)\right)(x,y;t)=
\sum_{\ell=1}^{\infty}(-1)^{\ell}(f)^{\ast(\ell+1)}(x,y;t).\label{eq: convolution with inf sum}
\end{equation}
\item
\begin{equation}\label{eq:series_bound}
\sum_{\ell=1}^{\infty}\left|(f)^{\ast\ell}(x,y;t)\right|
=O(t^{k})
\,\,\,\,\,
\text{\rm as $t\to 0$;}
\end{equation}
where the implied constant is independent of $x,y\in \X$;
\item
    $F(\cdot,y;t)\in L^\infty(\mu)$ for $\mu$ almost all $y\in\X$, $t\in[0,t_0]$.
\end{enumerate}
\end{lemma}

\begin{proof} Throughout the proof our discussion applies to $\mu$ almost all $x$ and $y$.  Let $A$ be an arbitrary compact subset of $\X\times \X\times\mathbb{R}_{\geq0}$.
Let $t_0>0$ be such that $A\subseteq \X\times \X\times [0,t_0]$.
We apply Lemma \ref{lem:convolution bounds}, with $p=\infty$, $q=1$ to get that
\begin{equation*}
  \label{eq:4}
\left|(f\ast f)(x,y;t)\right| \leq C \|f\|_{1,\mu}\frac{t^{k+1}}{(k+1)!}, \,\,\,\,\,
\text{\rm for all $x,y\in \X$ and $0<t<t_{0}$.}
\end{equation*}
Similarly, by induction for $\ell\geq 1$ we have the bound that
\begin{equation}
  \label{eq:mult_convolution_bound}
\left|(f^{\ast\ell})(x,y;t)\right|
\leq
C\|f\|_{1,\mu}^{\ell-1}\frac{t^{k+\ell-1}}{(k+\ell-1)!}, \,\,\,\,\,
\text{\rm for all $x,y\in \X$ and $0<t<t_{0}$.}
\end{equation}
The assertion regarding the convergence of (\ref{eq: series over ell}) now follows
from the Weierstrass criterion and the fact that $A\subseteq \X\times \X\times (0,t_0]$.

Fix $t>0$. The series (\ref{eq: series over ell}) converges absolutely.  When viewed as
a function of $x$, for any arbitrary but fixed $y$ and for any $0<t<t_0$, the series belongs to $L^{\infty}(\mu)$. Therefore,
\begin{equation}\label{eq. conv intergcange int sum}
\left(f\ast\left(\sum_{\ell=1}^{\infty}(-1)^{\ell}(f)^{\ast\ell}\right)\right)(x,y;t)= \int\limits_0^t \int\limits_{\X} f(x,z;t-r) \left( \sum_{\ell=1}^{\infty} (-1)^\ell (f)^{\ast \ell}(z,y;r)\right) \mu(dz)  dr.
\end{equation}
From the bound \eqref{eq:mult_convolution_bound}, when combined with the H\" older inequality with $p=1$ and $q=\infty$, we have
for an arbitrary $t\in(0,t_0)$ and $0<r<t$ that
\begin{align*}
\sum_{\ell=1}^{\infty} \int\limits_{\X} \left| f(x,z;t-r) (f)^{\ast \ell}(z,y;r)\right|\mu(dz)& \leq \sum_{\ell=1}^{\infty}\|f\|_{1,\mu} \cdot  C\|f\|_{1,\mu}^{\ell-1}\frac{t^{k+\ell-1}}{(k+\ell-1)!}\\ &= C\|f\|_{1,\mu}t^k \exp(t \|f\|_{1,\mu}).
\end{align*}
Hence we may interchange the sum over $\ell$ with the integral over $\X$ in \eqref{eq. conv intergcange int sum}. The above bound, combined with the Lebesgue Dominated Convergence Theorem, also proves statement (3), upon replacing $\ell$ by $\ell-1$.

By reasoning analogously, one easily shows that we may interchange the infinite sum over $\ell$ with the integral from $0$ to $t$ to deduce that
$$
\left(f\ast\left(\sum_{\ell=1}^{\infty}(-1)^{\ell}(f)^{\ast\ell}\right)\right)(x,y;t)= \sum_{\ell=1}^{\infty} (-1)^\ell \int\limits_0^t \int\limits_{\X} f(x,z;t-r)(f)^{\ast \ell}(z,y;r)\mu(z)  dr,
$$
which proves \eqref{eq: convolution with inf sum}.

Finally, the bound \eqref{eq:mult_convolution_bound}, and the fact that the series \eqref{eq: series over ell} converges absolutely on $\X\times \X\times (0,t_0]$, yields that
$$
\sum_{\ell=1}^{\infty}\left|(f)^{\ast\ell}(x,y;t)\right|\leq C\|f\|_{1,\mu}t^k \exp(t \|f\|_{1,\mu}),
$$
which proves \eqref{eq:series_bound}.
\end{proof}

Lemma \ref{lem:convergence of the series}, which is based on the H\"older conjugate pair $(1,\infty)$,
has an $L^2$-space version, which is the following Lemma.

\begin{lemma}\label{lemma 2,2, conv series}
  Let $f=f(x,y;t):\X\times \X\times\mathbb{R}_{>0}\to\mathbb{R}$.
Assume that for  $\mu$ almost all $x,y\in \X$ and all $t_{0} \in \mathbb{R}_{\geq0}$, the function $f(x,y;\cdot)$ is integrable on
 $[0,t_{0}]$ and $f(x,\cdot;t), f(\cdot, y;t) \in L^2(\mu)$ for all $t\in [0,t_0]$ and  $\mu$ almost all $x,y\in\X$.
Assume further that for all $t_{0} \in \mathbb{R}_{>0}$ there exist constants $C_1$, $C_2$ depending only upon $t_0$, $f$ and integer $k\geq 0$ such that
$$
||f(x,\cdot;t)||_{2,\mu}\leq C_1h(x)t^{k}
\,\,\,\,
\text{\rm and}
\,\,\,\,
 ||f(\cdot,y;t)||_{2,\mu}\leq C_2
\,\,\,\,\,
\text{\rm for  $\mu$ almost all $x,y\in \X$ and $0<t<t_{0}$,}
$$
for some function $h\in L^2(\mu)$ which is assumed to be bounded on compact subsets .
Then the series
\begin{equation}\label{eq: series over ell2}
F(x,y;t):=\sum_{\ell=1}^{\infty}(-1)^{\ell}(f)^{\ast\ell}(x,y;t)
\end{equation}
converges absolutely and uniformly on every compact subset of $\X\times \X\times\mathbb{R}_{\geq0}$.
In addition, we have the following results:
\begin{enumerate}
\item
\begin{equation}
\left(f\ast\left(\sum_{\ell=1}^{\infty}(-1)^{\ell}(f)^{\ast\ell}\right)\right)(x,y;t)=
\sum_{\ell=1}^{\infty}(-1)^{\ell}(f)^{\ast(\ell+1)}(x,y;t);\label{eq: convolution with inf sum2}
\end{equation}
\item
\begin{equation}\label{eq:series_bound2}
\sum_{\ell=1}^{\infty}\left\|(f)^{\ast\ell}(\cdot,y;t)\right\|_{2,\mu}
=O(t^{k})
\,\,\,\,\,
\text{\rm as $t\to 0$;}
\end{equation}
where the implied constant is independent of $x,y\in \X$;
\item $F(\cdot, y;t) \in L^2(\mu)$ for  $\mu$ almost all $y\in\X$ and $t\in[0,t_0]$.
\end{enumerate}
\end{lemma}

\begin{proof} Throughout the proof our discussion applies to $\mu$ almost all $x$ and $y$.
Let $A$ be an arbitrary compact subset of $\X\times \X\times\mathbb{R}_{\geq0}$.
Let $t_0>0$ be such that $A\subseteq \X\times \X\times [0,t_0]$.
We apply Lemma \ref{lem:convolution bounds}, with $p,\, q=2$ and $\ell=0$ to get
\begin{equation*}
\left|(f\ast f)(x,y;t)\right| \leq C_1C_2 h(x) \frac{t^{k+1}}{(k+1)!} \,\,\,\,\,
\text{\rm for all $x,y\in \X$ and $0<t<t_{0}$.}
\end{equation*}
In view of the fact that $h\in L^2(\mu)$, we see that $(f\ast f)(.,y;t)\in L^2(\mu)$ for all $y\in\X$. Moreover,
$$
\|(f\ast f)(\cdot,y;t)\|_{2,\mu}\leq C_1C_2 \|h\|_{2,\mu}  \frac{t^{k+1}}{(k+1)!}.
$$
Applying Lemma \ref{lem:convolution bounds}, with $p,\, q=2$, $\ell=k+1$ and constant $C_2$ equal to $   \frac{C_1C_2 \|h\|_{2,\mu}}{(k+1)!}$ we get
$$
\left|(f^{\ast 3})(x,y;t)\right|\leq C_1^2C_2\|h\|_{2,\mu} h(x)\frac{k!}{(2k+2)!}t^{2k+2}.
$$
Similarly, by induction for $\ell\geq 2$ we have the bound that
\begin{equation}
  \label{eq:mult_convolution_bound2}
\left|(f^{\ast\ell})(x,y;t)\right|
\leq
C_1^{\ell-1}C_2\|h\|_{2,\mu}^{\ell-2} h(x)\frac{(k!)^{\ell-2}}{((\ell-1)(k+1))!}t^{(\ell-1)(k+1)},
\end{equation}
for all $x,y\in \X$ and $0<t<t_{0}$.
The assertion regarding the convergence of (\ref{eq: series over ell2}) now follows
from the Weierstrass criterion and the fact that $A\subseteq \X\times \X\times (0,t_0]$.

Moreover, for $\ell\geq 2$ and all $y\in\X$ we have
\begin{equation}\label{eq. ell conv norm 2 bound}
\left\|(f^{\ast\ell})(\cdot,y;t)\right\|_{2,\mu}
\leq
C_1^{\ell-1}C_2\|h\|_{2,\mu}^{\ell-1} \frac{(k!)^{\ell-2}}{((\ell-1)(k+1))!}t^{(\ell-1)(k+1)}.
\end{equation}
This proves part (3) and also yields that
$$
\sum_{\ell=1}^{\infty}\left\|(f)^{\ast\ell}(\cdot,y;t)\right\|_{2,\mu}=O(t^k),
$$
as $t\downarrow 0$.

Fix $t>0$. The series (\ref{eq: series over ell2}) converges absolutely.  When viewed as
a function of $x$, for any arbitrary but fixed $y$ and for any $0<t<t_0$, in view of the bound \eqref{eq. ell conv norm 2 bound}, for any $\ell\geq 2$, the series belongs to $L^{2}(\mu)$. Therefore,
\begin{equation}\label{eq. conv intergcange int sum2}
\left(f\ast\left(\sum_{\ell=1}^{\infty}(-1)^{\ell}(f)^{\ast\ell}\right)\right)(x,y;t)= \int\limits_0^t \int\limits_{\X} f(x,z;t-r) \left( \sum_{\ell=1}^{\infty} (-1)^\ell (f)^{\ast \ell}(z,y;r)\right) \mu(dz)  dr.
\end{equation}
From the bound \eqref{eq. ell conv norm 2 bound} and the H\" older inequality with $p=q=2$, we have, for an arbitrary $t\in(0,t_0)$ and $0<r<t$ that
\begin{align*}
\sum_{\ell=1}^{\infty} \int\limits_{\X} \left| f(x,z;t-r) (f)^{\ast \ell}(z,y;r)\right|\mu(dz)& \leq C_2\sum_{\ell=1}^{\infty}C_1^{\ell}\|h\|_{2,\mu}^{\ell} \frac{(k!)^{\ell-1}}{(\ell(k+1))!}t^{\ell(k+1)}\\ &\leq \tilde C t^k \exp(tC_1\|h\|_{2,\mu} ).
\end{align*}
Therefore, we may interchange the sum over $\ell$ with the integral over $\X$ in \eqref{eq. conv intergcange int sum2}.
By reasoning analogously, one easily shows that we may interchange the infinite sum over $\ell$ with the integral from $0$ to $t$ to deduce that
$$
\left(f\ast\left(\sum_{\ell=1}^{\infty}(-1)^{\ell}(f)^{\ast\ell}\right)\right)(x,y;t)= \sum_{\ell=1}^{\infty} (-1)^\ell \int\limits_0^t \int\limits_{\X} f(x,z;t-r)(f)^{\ast \ell}(z,y;r)\mu(z)  dr,
$$
which proves \eqref{eq: convolution with inf sum2}.

With all this, the proof is complete.
\end{proof}

\section{Properties of the generalized time convolution on a Hilbert space}

\subsection{Generalized time convolution on a Hilbert space}

The generalized time convolution on any Hilbert space $\mathcal{H}$ of real valued functions on $\X$ with the inner product
$\langle \cdot, \cdot \rangle_{\mathcal{H}}$ is given by the following definition.

\begin{definition} \label{def. convHilb}\textbf{Time convolution on a Hilbert space} $\mathcal{H}$. Let $F_1,F_2: \X\times\X\times[0,\infty]$ be such that $F_1(x,\cdot;t), F_2(\cdot,y;t)\in \mathcal{H}$ for all $x,y\in\X$, $t\geq 0$ and such that the inner product
$$
\langle F_1(x,\cdot;\tau), F_2(\cdot,y;t-\tau) \rangle_{\mathcal{H}}
$$
is  integrable in $\tau$ on sets $[0, t]$, for all $t>0$, $x,y\in\X$. Then, time convolution on a Hilbert space $\mathcal{H}$ is defined as
\begin{equation}\label{eq. defn conv F1,2 Hilbert}
(F_1\ast F_2)(x,y;t):=\int\limits_{0}^t\langle F_1(x,\cdot;t-\tau),F_2(\cdot,y;\tau)\rangle_{\mathcal{H}} d\tau.
\end{equation}
\end{definition}


In this section we state lemmas which are Hilbert space $\mathcal{H}$ analogues of Lemmas \ref{lem:convolution bounds} and \ref{lemma 2,2, conv series}. Throughout this section we fix an orthonormal basis $\{\varphi_j\}_{j\geq 1}$, $\varphi_j:\X\to\mathbb{R}$
for the Hilbert space $\mathcal{H}$ of real-valued functions on $\X$.

\begin{lemma}
\label{lem:convolution bounds Hilb} Let $F_{1},F_{2}:\X\times \X\times\mathbb{R}_{>0}\to\mathbb{R}$
be as in Definition \ref{def. convHilb}.  Assume that $a_j(x,t):= \langle F_1(x,\cdot; t), \varphi_j(\cdot) \rangle_{\mathcal{H}}$ is
continuous in $t$ and that
$$
\sum_{j\geq 1}a_j(x,t)^2 =||F_1(x,\cdot,t)||_{\mathcal{H}}^2
$$
is also continuous in $t$ for all $x\in\X$. For some $t_{0} > 0$, assume there exist constants $C_1, C_2$ and integers $k,\ell\geq 0$
such that for all $0<t<t_{0}$ and all $x,y \in \X$, we have that $a_j(\cdot ;t)\in\mathcal{H}$, for all $j\geq 1$,  and
$$
\sum_{j\geq 1} ||a_j(\cdot,t)||^2_{\mathcal{H}}\leq C_{1}t^{k}
  \,\,\,\,\,
  \text{\rm and}
  \,\,\,\,\,
  \left\|F_{2}(\cdot, y;t)\right\|_{\mathcal{H}}^2\leq C_{2}t^{\ell} \,\, \text{\rm for all $y\in\X$,}
$$
where $||\cdot||_{\mathcal{H}}$ denotes the norm induced by the inner product on $\mathcal{H}$. Then, for all $y\in \X$ and $0<t<t_{0}$
we have that $F_1\ast F_2(\cdot,y; t)\in\mathcal{H}$ and
\[
||(F_{1}\ast F_{2})(\cdot,y;t)||_{\mathcal{H}}^2 \leq C_{1}C_{2}\frac{1}{(k+1)(\ell+1)}
  t^{k+\ell+2}
  \,\,\,\,\,
  \text{\rm for $0<t<t_{0}$.}
\]
\end{lemma}

\begin{remark}
Note that the conditions posed on $F_1$ are equivalent to saying that it defines a Hilbert-Schmidt
type operator mapping from $\mathcal{H}$ to $\mathcal{H}$.
\end{remark}

\begin{proof} Let
$$
F_2(\cdot, y;\tau)=\sum_{j\geq 1} b_j(y;\tau)\varphi_j(\cdot)
$$
be an expansion of $F_2$ in terms of the basis of $\mathcal{H}$ where $0<\tau<t_0$ and $y\in\X$. Then, for $0<\tau<t<t_0$ and $x,y\in\X$ we have that
  $$
  \langle F_1(x,\cdot;t-\tau),F_2(\cdot,y;\tau)  \rangle_{\mathcal{H}}=\sum_{j\geq 1}a_j(x;t-\tau)b_j(y;\tau).
  $$
  We claim that, when viewed as a function of $x\in\X$ the pairing $\langle F_1(x,\cdot;t-\tau),F_2(\cdot,y;\tau)  \rangle_{\mathcal{H}}$
  is in $\mathcal{H}$ for all $0<\tau<t<t_0$ and $y\in\X$.
To see this, choose any $j, \ell\geq 1$ and let $c_{j\ell}(t-\tau):=\langle a_j(x;t-\tau),\varphi_\ell(x)  \rangle_{\mathcal{H}}$. Then
$$
 \langle \sum_{j\geq 1}a_j(\cdot;t-\tau)b_j(y;\tau) , \varphi_\ell(\cdot)  \rangle_{\mathcal{H}} = \sum_{j\geq 1} c_{j\ell}(t-\tau)b_j(y;\tau),
$$
so then
\begin{align*}
|| \langle F_1(x,\cdot;t-\tau),F_2(\cdot,y;\tau)  \rangle_{\mathcal{H}}||_{\mathcal{H}}^2 &=\sum_{\ell\geq 1}\left( \sum_{j\geq 1} c_{j\ell}(t-\tau)b_j(y;\tau) \right)^2 \\&\leq \sum_{\ell\geq 1} \left( \sum_{j\geq 1} c_{j\ell}(t-\tau)^2 \sum_{j\geq 1}b_j(y;\tau)^2 \right)\\&\leq \left\|F_{2}(\cdot, y;\tau)\right\|_{\mathcal{H}}^2 \sum_{j\geq 1}  \sum_{\ell\geq 1} c_{j\ell}(t-\tau)^2\\& = \left\|F_{2}(\cdot, y;\tau)\right\|_{\mathcal{H}}^2  \sum_{j\geq 1} ||a_j(\cdot,t-\tau)||^2_{\mathcal{H}}\\&\leq C_1C_2 \tau^\ell (t-\tau)^k.
\end{align*}
This proves that $\langle F_1(x,\cdot;t-\tau),F_2(\cdot,y;\tau)  \rangle_{\mathcal{H}}\in \mathcal{H}$.
Continuing, we claim that
\begin{equation}\label{eq. int of inner prod}
\int_0^t
\langle F_1(x,\cdot;t-\tau),F_2(\cdot,y;\tau)  \rangle_{\mathcal{H}} d\tau = \sum_{j\geq 1} \int_0^t a_j(x;t-\tau)b_j(y;\tau)d\tau.
\end{equation}
Indeed, it is immediate that
\begin{align*}
\sum_{j\geq 1} \int_0^t | a_j(x;t-\tau)b_j(y;\tau)| d\tau &=\int_0^t \left(\sum_{j\geq 1}| a_j(x;t-\tau)b_j(y;\tau)|\right) d\tau \\&\leq \int_0^t  ||F_1(x, \cdot;t)||_{\mathcal{H}} ||F_2(\cdot,y;t)||_{\mathcal{H}} <\infty,
\end{align*}
which justifies the interchange of the sum and integral in \eqref{eq. int of inner prod}.

Let
$$
d_{\ell}(y):=\langle \int_0^t
\langle F_1(x,\cdot;t-\tau),F_2(\cdot,y;\tau)  \rangle_{\mathcal{H}} d\tau, \varphi_\ell (x)\rangle_{\mathcal{H}}= \sum_{j\geq 1} \langle  \int_0^t  a_j(\cdot;t-\tau)b_j(y;\tau) d\tau ,  \varphi_\ell(\cdot) \rangle_{\mathcal{H}},
$$
where the last equality follows from \eqref{eq. int of inner prod}.
But, using the same argument as above to justify the interchange of the sum over $m$ and the integral, we have that
$$
\int_0^t  a_j(\cdot;t-\tau)b_j(y;\tau) d\tau = \sum_{m\geq 1} \varphi_m(\cdot) \int_0^t  c_{jm}(t-\tau)b_j(y;\tau) d\tau,
$$
hence
$$
d_{\ell}(y)=\sum_{j\geq 1}  \int_0^t  c_{j\ell}(t-\tau)b_j(y;\tau) d\tau.
$$
Using the Cauchy-Schwarz inequality, followed by the H\"older inequality we get
\begin{align*}
|d_\ell(y)|&\leq \int_0^t \sum_{j\geq 1}  |c_{j\ell}(t-\tau)b_j(y;\tau)| d\tau\\
&\leq \int_0^t \left(\sum_{j\geq 1}c_{j\ell}(t-\tau)^2 \right)^{1/2} \left(\sum_{j\geq 1} b_j(y;\tau)^2 \right)^{1/2} d\tau\\
&\leq \left( \int_0^t \left(\sum_{j\geq 1}c_{j\ell}(t-\tau)^2 \right)d\tau\right)^{1/2}
 \left( \int_0^t \left(\sum_{j\geq 1}b_j(y;\tau)^2 \right)d\tau\right)^{1/2}.
\end{align*}
Therefore
\begin{align*}
||(F_{1}\ast F_{2})(\cdot,y;t)||_{\mathcal{H}}^2 &=  \left\| \int_0^t
\langle F_1(x,\cdot;t-\tau),F_2(\cdot,y;\tau)  \rangle_{\mathcal{H}} d\tau\right\|_{\mathcal{H}}^2 \\&=\sum_{\ell\geq 1} d_\ell^2(y)
\\ &\leq \sum_{\ell\geq 1} \int_0^t \left(\sum_{j\geq 1}c_{j\ell}(t-\tau)^2 \right)d\tau \cdot C_2  \int_0^t \tau^\ell d\tau \\&= C_2\frac{t^{\ell+1}}{\ell+1}  \int_0^t \left( \sum_{j\geq 1} ||a_j(\cdot,t-\tau)||^2_{\mathcal{H}}\right) d\tau \\&\leq C_{1}C_{2}\frac{1}{(k+1)(\ell+1)}   t^{k+\ell+2},
\end{align*}
where we used that
$$
\sum_{j\geq 1}b_j(y;\tau)^2 =||F_2(\cdot,y;\tau)||_\mathcal{H}^2\leq C_2\tau^\ell.
$$
With all this, the proof of Lemma \ref{lem:convolution bounds Hilb} is complete.
\end{proof}

Next, we would like to define an $\ell$-fold generalized time convolution of a function $f$ on $\mathcal{H}$.
To do so, we assume that $f=f(x,y;t):\X\times \X\times\mathbb{R}_{>0}\to\mathbb{R}$ possesses the following property.  For any  $T>0$, all $t\in (0,T]$,
and all arbitrary but fixed $x,y\in \X$, the functions $f(\cdot, y;t): \X \to \mathbb{R} $ and $f(x,\cdot;t): \X \to \mathbb{R} $
belong to $\mathcal{H}$ and the function $f(x,y;\cdot) :(0,T] \to \mathbb{R}$ is integrable.
For any positive integer $\ell$ and any such function $f$, we can inductively define the $\ell$-fold
convolution $(f)^{\ast\ell}(x,y;t)$ for $t\in(0,T]$ by setting $(f)^{\ast1}(x,y;t)=f(x,y;t)$
and, for $\ell\geq2$ we denote that
$$
(f)^{\ast\ell}(x,y;t):=\left(f\ast(f)^{\ast(\ell-1)}\right)(x,y;t),
$$
under additional assumption that $(f)^{\ast(\ell-1)}(\cdot,y;t)\in \mathcal{H}$ for all $t\in(0,T]$ and all $\ell \geq 2$.

With this definition, we have the following Hilbert space analogue of Lemma \ref{lemma 2,2, conv series}.

\begin{lemma}\label{lemma 2,2, conv series Hilb}
  Let $f=f(x,y;t):\X\times \X\times\mathbb{R}_{>0}\to\mathbb{R}$.
Assume that for all $x,y\in \X$ and all $t_{0} \in \mathbb{R}_{\geq0}$, we have $f(x,\cdot;t), f(\cdot, y;t) \in \mathcal{H}$ for all $t\in [0,t_0]$ and all $x,y\in\X$. Let  $a_j(x,t):= \langle f(x,\cdot; t), \varphi_j(\cdot) \rangle_{\mathcal{H}}$ be continuous in $t$ variable and such that $\sum_{j\geq 1}a_j(x,t)^2 =||f(x,\cdot,t)||_{\mathcal{H}}^2$ is also continuous in $t$ for all $x\in\X$. Assume further that  there exist constants $C_1, C_2$ and integer $k\geq 0$
such that for all $0<t<t_{0}$ and all $x,y \in \X$, we have that $a_j(\cdot ;t)\in\mathcal{H}$, for all $j\geq 1$ and
$$
\sum_{j\geq 1} ||a_j(\cdot,t)||^2_{\mathcal{H}}\leq C_{1}
  \,\,\,\,\,
  \text{\rm and}
  \,\,\,\,\,
  \left\|f(\cdot, y;t)\right\|_{\mathcal{H}}^2\leq C_{2}t^{k+2}.
$$
Then the $\ell$-fold
convolution $(f)^{\ast\ell}(x,y;t)$ for $t\in(0,t_0]$ is well defined for all $\ell\geq 2$ and
\begin{equation}\label{eq:series_bound2 Hilb}
\sum_{\ell=1}^{\infty}\left\|(f)^{\ast\ell}(\cdot,y;t)\right\|_{\mathcal{H}}^2
=O(t^{k})
\,\,\,\,\,
\text{\rm as $t\to 0$,}
\end{equation}
where the implied constant is independent of $x,y\in \X$.
Therefore, the series  \eqref{eq: series over ell2}
converges in the Hilbert space norm for $(x,y;t)$ in a compact subset of $\X\times \X\times\mathbb{R}_{\geq0}$. Moreover,
\eqref{eq: convolution with inf sum2} holds true and $F(\cdot, y;t) \in \mathcal{H}$ for all $y\in\X$ and $t\in[0,t_0]$.
\end{lemma}
\begin{proof}
We apply Lemma \ref{lem:convolution bounds Hilb}
with $F_1=F_2=f$ and $\ell =k $ to deduce that $f\ast f(\cdot, y;t) \in \mathcal{H}$ for all $t\in[0,t_0]$ and all $y\in\X$.
Furthermore,
$$
||(f\ast f)(\cdot,y;t)||_{\mathcal{H}}^2 \leq C_{1}C_{2}\frac{1}{(k+1)}
  t^{k+2}
  \,\,\,\,\,
  \text{\rm for $0<t<t_{0}$.}
$$
Next, we apply Lemma \ref{lem:convolution bounds Hilb}
with $F_1=f$, $F_2=f\ast f$ and $\ell =k+2 $ to deduce that $ f^{\ast (3)}(\cdot, y;t) \in \mathcal{H}$ for all $t\in[0,t_0]$ and all $y\in\X$ and that
$$
|| f^{\ast 3}(\cdot,y;t)||_{\mathcal{H}}^2 \leq C_{1}^2 C_{2}\frac{1}{(k+1)(k+3)}
  t^{k+4}
  \,\,\,\,\,
  \text{\rm for $0<t<t_{0}$.}
$$
Proceeding inductively, for $\ell\geq 3$ we get
$$
|| f^{\ast \ell}(\cdot,y;t)||_{\mathcal{H}}^2 \leq C_{1}^{\ell-1} C_{2}\frac{1}{(k+1)\ldots (k+2\ell-3)}
  t^{k+2(\ell -1)}
  \,\,\,\,\,
  \text{\rm for $0<t<t_{0}$.}
$$
Therefore,
$$
\sum_{\ell=1}^{\infty}\left\|(f)^{\ast\ell}(\cdot,y;t)\right\|_{\mathcal{H}}^2\leq t^k C_2 \sum_{\ell=0}^{\infty}\frac{(C_1t^2)^{\ell}}{\ell!}=O(t^{k})
\,\,\,\,\,
\text{\rm as $t\to 0$,}
$$
which proves the first claim. The Cauchy-Schwartz inequality implies that
$$
F(\cdot, y;t):=\sum_{\ell=1}^{\infty}(-1)^{\ell}(f)^{\ast\ell}(\cdot, y;t)\in\mathcal{H}
$$
for all $0<t<t_{0}$ and
$$
|| F(\cdot, y;t)||_{\mathcal{H}}^2 \leq \widetilde{C} t^k,
$$
for $0<t<t_0$ and some constant $\widetilde C$.
Finally, continuity of the inner product and the bounds for the $\ell$- fold convolution yield that
   \begin{align*}\label{eq. conv intergcange int sum2}
\left(f\ast\left(\sum_{\ell=1}^{\infty}(-1)^{\ell}(f)^{\ast\ell}\right)\right)(x,y;t)&= \int\limits_0^t \langle f(x,\cdot;t-r), \left( \sum_{\ell=1}^{\infty} (-1)^\ell (f)^{\ast \ell}(\cdot,y;r)\right) \rangle_{\mathcal{H}}   dr\\&= \sum_{\ell=1}^{\infty} (-1)^\ell \int\limits_0^t \langle f(x,\cdot;t-r),(f)^{\ast \ell}(\cdot,y;r)\rangle_{\mathcal{H}}dr.
\end{align*}
With all this, the proof of Lemma \ref{lemma 2,2, conv series Hilb} is complete.
\end{proof}

\section{Parametrix construction of the heat kernel in three different settings}\label{sec. param. constr}

In this section we describe the construction of the heat kernel using the parametrix approach in three settings. We will first describe the construction
when the underlying space is a measurable space $(\X,\mathcal{B},\mu)$ and the conjugate pairs $p,q$ in the parametrix construction are
$(1,\infty)$ and then $(2,2)$, respectively. Then, we will describe the construction when the underlying space is a separable
Hilbert space $\mathcal{H}$.

A parametrix is a function that serves as an initial approximation of the heat kernel. It is chosen based on certain desirable properties
which are reasonably general. Often,
a parametrix is simpler to compute than the exact heat kernel. In essence, the parametrix construction provides a powerful and flexible framework
for studying heat kernels in various mathematical settings, particularly when dealing with complex geometric structures.
One then uses iterative techniques or formulas to improve the accuracy of the parametrix, gradually approaching the true heat kernel.
As such, the heat kernel can be viewed as a type of ``fixed-point'' through a functional-analytic process.

\medskip
\subsection{A $(1,\infty)$ parametrix construction of the heat kernel}

The heat operator $L$ is defined as
$$
L:= \Delta+\partial_t
$$
where $\Delta$ is the Laplacian defined in Section \ref{sec. setup} and $\partial_t$ is the first partial derivative in time-variable $t$.
Of course, the normalized/probabilistic heat operator $\tilde L$ is defined similarly using
 the normalized/probabilistic Laplacian.  As usual, the quantification \it all $x \in \X$ \rm means \it $\mu$-almost all $x \in \X$. \rm

\begin{definition}\label{def. parametrix1} (\textbf{A $(1,\infty)$ parametrix.})
Let $k\geq 0$ be an integer. A \emph{$(1,\infty)$ parametrix} $H$ of order $k$  for the heat operator $L$ on $(\X,\mu)$ is any continuous function $H=H(x,y;t):\X\times \X\times (0,\infty)$
which is smooth in the time variable $t\in(0,\infty)$, $\mu$-integrable in each space variable, and satisfies the following properties.
\begin{enumerate}
\item The function $L_{x}H(x,y;t)$ extends to a continuous function on  $\X\times \X\times [0,\infty)$.

 \item  For all  $x\in \X$ and $t>0$ we have that $\Delta_{x}H(x,\cdot;t)$ and $\partial_t  H(x,\cdot;t)$ are in $L^1(\mu)$.
\item For any $t_0>0$ there exists a constant $C=C(t_0)$, depending only on
$t_0$, such that
$$
|L_{x}H(x,y;t)| \leq C(t_0)t^k
\,\,\,\,\,
\text{\rm for $t \in (0, t_0]$}
\,\,\,
\text{\rm and all}
\,\,\,
\text{\rm $x,y\in \X$.}
$$
\item For all $x\in \X$ and all $f\in L^\infty(\mu)$
\begin{equation}\label{eq:dirac property of heat kernel}
\lim\limits_{t \rightarrow 0^+}
\int\limits_{\X}H(x,y;t)f(y)\mu(dy)=f(x)
\,\,\,\,\,\text{\rm for all $x\in\X$.}
\end{equation}
\end{enumerate}
\end{definition}
Note that the third assumption on the parametrix $H$ implies that $L_{x}H(x,\cdot;t)\in L^1(\mu)$. Moreover, we use the subscript $x$ in the heat $L_x$ to denote that the Laplacian is acting in the first spatial variable.

\begin{remark}
  A $(1,\infty)$ parametrix for the heat operator $\tilde L$ is defined in the analogous way by
  replacing operators $\Delta$ and $L$ by $\tilde \Delta$ and $\tilde L$, respectively.
\end{remark}

\begin{lemma}
\label{lem: L of convol 1, infty}
Let $H$ be a $(1,\infty)$ parametrix for the heat operator on $(\X,\mu)$ of any order $k$.
Let $f=f(x,y;t):\X\times \X\times\mathbb{R}_{> 0}\to\mathbb{R}$
be a continuous function in $t$ for all $x,y\in \X$. Assume further that for all $t>0$ and all $y\in\X$, the function
$f(\cdot,y;t)$, when viewed as a function on $\X$, belongs to $L^\infty(\mu)$.  Then
$$
L_{x}( H\ast f)(x,y;t)=f(x,y;t)+(L_{x}H \ast f)(x,y;t)
$$
for all $x,y\in \X$ and $t\in\mathbb{R}_{>0}$.
\end{lemma}

\begin{proof} As stated we are assuming that $f(\cdot, y;t)\in L^{\infty}(\mu)$ and $H(x,\cdot;t), \, \Delta_{x}H(x,\cdot;t)\in L^1(\mu)$,
for all $(x,y;t) \in \X\times \X \times \mathbb{R}_{>0}$.  Therefore, the convolutions $H\ast f$ and $(\Delta_{x}H)\ast f$ are well defined. Moreover,
\begin{align}\notag
\Delta_{x}(H\ast f)(x,y;t)&= \int\limits_{\X}\left(\int\limits _{0}^{t}\int\limits_{\X} \left(H(x,w;t-r) - H(z,w;t-r)\right)f(w,y;r)\mu(d w)dr\right)\rho_x(dz)
    \\ &= \int\limits _{0}^{t} \int\limits_{\X} \left(\int\limits_{\X} \left(H(x,w;t-r) - H(z,w;t-r)\right)\rho_x(dz)\right) f(w,y;r)\mu( dw)  dr \notag\\
    &= (\Delta_{x}H)\ast f (x,y;t).\label{eq. change of delta and ast}
\end{align}
In the above integral the application of Fubini-Tonelli theorem is justified by the fact that the H\"older inequality implies that
\begin{align*}
\int\limits_{\X}\int\limits _{0}^{t}\int\limits_{\X} &\left|H(x,w;t-r) - H(z,w;t-r)\right|\cdot |f(w,y;r)|\mu( dw)dr\rho_x(dz) \\ & \leq \int\limits_{\X}\int\limits _{0}^{t} \left(\| H(x,\cdot;t-r)\|_{L^1(\mu)}+ \| H(z,\cdot;t-r)\|_{L^1(\mu)}\right)  \| f(\cdot,y;r)\|_{L^\infty(\mu)}dr\rho_x(dz).
\end{align*}
The above integral is finite, because of assumption (2) posed on the parametrix $H$ and the fact that $\rho_x(\X)<\infty$.
Therefore, we have that
\begin{align}\label{eq:heat_of_convolution1}
L_{x}(H\ast f)(x,y;t)&=
    \frac{\partial}{\partial t}(H\ast f)(x,y;t)
    + \Delta_{x}(H\ast f)(x,y;t)\\ &=
  \frac{\partial}{\partial t}(H\ast f)(x,y;t)+  \left( (\Delta_{x}H)\ast f\right)(x,y;t)\nonumber.
\end{align}
The function
$$
\int\limits_{\X}H(x,z;t-r)f(z,y;r) \mu(dz)
$$
is continuous in the time variable.  To see this, note that the integrand is continuous in $t$. Then, one can apply
the H\"older inequality and the assumptions on the parametrix $H$ and $f$.
As a result, we can apply the
Leibniz integration formula.  Upon doing so, we get that
the first term on the right hand side of \eqref{eq:heat_of_convolution1} is equal to
\begin{align} \label{eq.deriv in t1}
\lim_{\epsilon\to0^+}\frac{\partial}{\partial t}&\int\limits _{0}^{t-\epsilon}\int\limits_{\X} H(x,z;t-r)f(z,y;r) \mu(dz)dr\\&
=\lim_{\epsilon\to0^+}\int\limits_{\X}H(x,z;\epsilon)f(z,y;t)\mu(dz)+\int\limits _{0}^{t} \frac{\partial}{\partial t} \int\limits_{\X}
H(x,z;t-r)f(z,y;r) \mu(dz)dr.\notag
\end{align}
The two assumptions that $f(\cdot,y;t)\in L^\infty(\mu)$ and that $\frac{\partial}{\partial t}  H(x,\cdot;t)\in L^1(\mu)$ combine,
together with the H\"older inequality and the theorem on differentiation of the Lebesgue integral with respect to parameter, to yield that
$$
\frac{\partial}{\partial t} \int\limits_{\X}
H(x,z;t-r)f(z,y;r) \mu(dz)= \int\limits_{\X}\frac{\partial}{\partial t} H(x,z;t-r)f(z,y;r) \mu(dz).
$$
In view of assumption (4) on the parametrix $H$, we get from \eqref{eq.deriv in t1} that
$$
\frac{\partial}{\partial t}\int\limits _{0}^{t}\int\limits_{\X}H(x,z;t-r)f(z,y;r) \mu(dz)dr= f(x,y;t) +
\left(\frac{\partial}{\partial t}H\ast f\right)(x,y;t).
$$
Therefore,
\begin{align*}
L_{x}(H\ast f)(x,y;t) &=
  \frac{\partial}{\partial t}(H\ast f)(x,y;t)+   (\Delta_{x}H\ast f)(x,y;t)
\\&=f(x,y;t)+\left(\frac{\partial}{\partial t}H\ast f\right)(x,y;t)+(\Delta_{x}H\ast f)(x,y;t)
\\&=f(x,y;t)+(L_{x}H\ast f)(x,y;t),
\end{align*}
as claimed.
\end{proof}

With all this, we now can state and prove the main theorem in this subsection.

\begin{theorem}
\label{thm: formula for heat kernel} Let $H$ be a $(1,\infty)$ parametrix of
order $k\geq0$ for the heat operator on $\X$.  For $x,y\in \X$ and $t\in \mathbb{R}_{\geq 0}$, let
\begin{equation}\label{eq:Neuman_series}
F(x,y;t):=\sum_{\ell=1}^{\infty}(-1)^{\ell}
(L_{x}H)^{\ast\ell}(x,y;t).
\end{equation}
Then the Neumann series \eqref{eq:Neuman_series}
converges absolutely and uniformly on every compact subset of $\X\times \X\times\mathbb{R}_{\geq 0}$.
Furthermore, the $L^1(\mu)$ heat kernel $K_{\X}$ on $\X$
associated to graph Laplacian $\Delta_{x}$ is given by
\begin{equation}\label{eq:heat_kernel_parametrix}
K_{\X}(x,y;t)=H(x,y;t)+(H\ast F)(x,y;t)
\end{equation}
and
\begin{equation}\label{eq. small t asympt}
(H\ast F)(x,y;t)
= O(t^{k+1})
\,\,\,\,\,
\text{\rm as $t \rightarrow 0^+$,}
\end{equation}
for all $x,y\in\X$.
\end{theorem}

\begin{proof}
Set
$$
\tilde{H}(x,y;t) := H(x,y;t)+(H\ast F)(x,y;t).
$$
From the assumptions on the $(1,\infty)$ parametrix $H$ it follows that the function $L_xH$ satisfies assumptions of Lemma \ref{lem:convergence of the series}. Therefore, the Neumann series \eqref{eq:Neuman_series} converges absolutely and uniformly on every compact subset of $\X\times \X\times\mathbb{R}_{\geq 0}$; it is continuous function of the time variable and belongs to $L^{\infty}(\mu)$, when viewed as a function of the first spatial variable $x$, for all $y\in\X$. Therefore,  $\tilde{H}(x,y;\cdot):[0,\infty)\to\mathbb{R}$ is smooth on $(0,\infty)$, continuous at $t=0$ and $\tilde{H}(\cdot,y;t)\in L^1(\mu)$ for all $y\in\X$, $t>0$.

 We want to show that
\begin{equation}\label{eq:heat_kernel_criteria 0}
L_{x}\tilde{H}(x,y;t)=0
\,\,\,\,\,
\text{\rm and}
\,\,\,\,\,
\lim_{t\to0^+}\int\limits_{\X}\tilde H(x,y;t)f(y)\mu(dy)=f(x)
\,\,\,\,
\text{\rm for all $x\in\X, f\in L^\infty(\mu)$.}
\end{equation}
By Lemma \ref{lem:convergence of the series}, the series $F(x,y;t)$
has order $O(t^{k})$ as $t \rightarrow 0^+$.  Since $H$ is in $L^1(\mu)$,
we can apply Lemma \ref{lem:convolution bounds} with $F_1=H$ and $F_2=F$ to obtain the asymptotic bound that
$$
(H\ast F)(x,y;t)= O(t^{k+1})
\,\,\,\,\,
\text{\rm as $t \rightarrow 0^+$.}
$$
Therefore, for any $f\in L^\infty(\mu)$
$$
\lim_{t\to 0^+} \int_{\X}\tilde{H}(x,y;t)f(y)\mu(dy) = \lim_{t\to 0^+}\int_{\X}H (x,y;0)f(y)\mu(dy)=f(x),
$$
by the assumption (4) for the $(1,\infty)$ parametrix $H$.

It remains to prove the vanishing of $L_{x}\tilde{H}$ in \eqref{eq:heat_kernel_criteria 0}. For this, we can apply Lemma \ref{lem: L of convol 1, infty} combined with Lemma \ref{lem:convergence of the series} to get that

\begin{align*}
L_{x}\tilde{H}(x,y;t)&=
L_{x}H(x,y;t)+L_{x}(H\ast F)(x,y;t)
\\&=
L_{x}H(x,y;t)+\sum_{\ell=1}^{\infty}(-1)^{\ell}(L_{x}
{H})^{\ast\ell}(x,y;t)+(L_{x}{H})*\left(\sum_{\ell=1}^{\infty}(-1)^{\ell}(L_{x}{H})^{\ast\ell}\right)(x,y;t)\\&=
L_{x}H(x,y;t)+\sum_{\ell=1}^{\infty}(-1)^{\ell}(L_{x}
{H})^{\ast\ell}(x,y;t)+\sum_{\ell=1}^{\infty}(-1)^{\ell}(L_{x}{H})^{\ast(\ell+1)}(x,y;t)\\&
=0.
\end{align*}
To be precise, in the above calculations we used
that absolute convergence of the series defining $F(x,y;t)$
in order to change the order of summation.  With all this, the proof is complete.
\end{proof}

\subsection{A $(2,2)$ parametrix construction of the heat kernel}

\begin{definition}\label{def. parametrix2} (\textbf{A $(2,2)$ parametrix.})
Let $k\geq 0$ be an integer. A \emph{($2,2$) parametrix} $H$ order $k$  for the heat operator $L$ on $(\X,\mu)$
is any continuous function $H=H(x,y;t):\X\times \X\times \mathbb{R}_{>0}$
which is smooth on $\mathbb{R}_{>0}$ in time variable $t$, belongs to $L^2(\mu)$ in each space variable, and satisfies the following properties.
\begin{enumerate}
\item The function $L_{x}H(x,y;t)$ extends to a continuous function of $t \in [0,\infty)$ and for all $x,y\in\X$.

 \item  For all  $x\in \X$ and $t\in \mathbb{R}_{>0}$, we have that $\Delta_{x}H(x,\cdot;t)$ and $\partial_t  H(x,\cdot;t)$
 are in $L^2(\mu)$. Moreover, $||H(x,\cdot;t)||_{2,\mu} $ is bounded by a $\rho_x$-integrable function for all $t\in \mathbb{R}_{>0}$.
 \footnote{Note that every bounded function is $\rho_x$-integrable.}
\item For any $t_0\in \mathbb{R}_{>0}$ there exists a constant $C=C(t_0)$, depending only on
$t_0$, and a function $h\in L^2(\mu)$ such that
$$
||L_{x}H(x,\cdot;t)||_{2,\mu} \leq C(t_0)h(x)t^k
\,\,\,\,\,
\text{and}\quad ||L_{x}H(\cdot,y;t)||_{2,\mu} \leq C(t_0)
$$
for $t \in (0, t_0]$ and all $x,y\in \X$.
\item For all $x\in \X$, and all $f\in L^2(\mu)$
\begin{equation}\label{eq:dirac property of heat kernel2}
\lim\limits_{t \rightarrow 0^+}\int\limits_{\X}H(x,y;t)f(y)\mu(dy)=f(x)
\,\,\,\,\,
\text{\rm for all $x\in\X$.}
\end{equation}
\end{enumerate}
\end{definition}

\begin{lemma}
\label{lem: L of convol}
Let $H$ be a $(2,2)$ parametrix for the heat operator on $(\X,\mu)$ of any order.
Let $f=f(x,y;t):\X\times \X\times\mathbb{R}_{> 0}\to\mathbb{R}$
be a continuous function in $t$ for all $x,y\in \X$. Assume further that for all $t\in \mathbb{R}_{> 0}$ the function
$f(\cdot,y;t)$ when viewed as a function on $\X$, for all $y\in\X$ belongs to $L^2(\mu)$.  Then
$$
L_{x}( H\ast f)(x,y;t)=f(x,y;t)+(L_{x}H \ast f)(x,y;t)
$$
for all $x,y\in \X$ and $t\in\mathbb{R}_{>0}$.
\end{lemma}

\begin{proof} As stated $f(\cdot, y;t)\in L^{2}(\mu)$ and $H(x,\cdot;t), \, \Delta_{x}H(x,\cdot;t)\in L^2(\mu)$
for all $(x,y;t) \in \X\times \X \times \mathbb{R}_{> 0}$.  Therefore, the convolutions $H\ast f$ and $(\Delta_{x}H)\ast f$ are well defined,
and we have that
\begin{align}\notag
L_{x}(H\ast f)(x,y;t)&=
    \frac{\partial}{\partial t}(H\ast f)(x,y;t)
    + \Delta_{x}(H\ast f)(x,y;t)\\ &=
  \frac{\partial}{\partial t}(H\ast f)(x,y;t)+  \left( (\Delta_{x}H)\ast f\right)(x,y;t),\label{eq:heat_of_convolution}
\end{align}
where the second equation is proved analogously as \eqref{eq. change of delta and ast}, namely by
applying the Fubini-Tonelli theorem, which is justified by the H\"older inequality.

The rest of the proof is analogous to the proof of Lemma \ref{lem: L of convol 1, infty}, so we refer the reader to that
proof for further details.
\end{proof}

With all this, we now can state the main theorem in this subsection.

\begin{theorem}
\label{thm: formula for heat kernel 2} Let $H$ be a $(2,2)$ parametrix of
order $k\geq0$ for the heat operator on $\X$.
Then the Neumann series \eqref{eq:Neuman_series}
converges absolutely and uniformly on every compact subset of $\X\times \X\times\mathbb{R}_{\geq 0}$.
Furthermore, the $L^2$-heat kernel $K_{\X}$ on $\X$
associated to graph Laplacian $\Delta_{x}$ is given by \eqref{eq:heat_kernel_parametrix} and the asymptotic \eqref{eq. small t asympt} holds true.
\end{theorem}

\begin{proof}
Set
$$
\tilde{H}(x,y;t) := H(x,y;t)+(H\ast F)(x,y;t).
$$
By Lemma \ref{lemma 2,2, conv series}, the series $F(x,y;t)$ defined in \eqref{eq:Neuman_series}
converges uniformly and absolutely. Moreover, according to \eqref{eq:series_bound2}, as a function of the first variable, it belongs to $L^2(\mu)$ and has order $O(t^{k})$ as $t \rightarrow 0$.  Since $H$ is in $L^2(\mu)$ as a function of the first variable, this, combined with Lemma \ref{lem:convolution bounds} with $p=q=2$ implies that $\tilde{H}(x,y;t)$ satisfies condition (1) in the definition of the $L^2$ heat kernel.

Next, we want to show that
\begin{equation}\label{eq:heat_kernel_criteria}
L_{x}\tilde{H}(x,y;t)=0
\,\,\,\,\,
\text{\rm and}
\,\,\,\,\,
\lim_{t\rightarrow 0^+}\int\limits_{\X}\tilde H(x,y;0)f(y)\mu(dy)=f(x)
\,\,\,\,\,
\text{\rm for all $x\in\X$ and $f\in L^2(\mu)$.}
\end{equation}
Combining Lemma \ref{lemma 2,2, conv series} with Lemma \ref{lem:convolution bounds} we get the asymptotic bound that
$$
(H\ast F)(x,y;t)= O(t^{k+1})
\,\,\,\,\,
\text{\rm as $t \rightarrow 0^+$.}
$$
Therefore, for any $f\in L^2(\mu)$
$$
\lim_{t\to 0^+} \int_{\X}\tilde{H}(x,y;t)f(y)\mu(dy) =\lim_{t\to 0^+} \int_{\X}H(x,y;t)f(y)\mu(dy)=f(x).
$$
It remains to prove the vanishing of $L_{x}\tilde{H}$ in \eqref{eq:heat_kernel_criteria}.
For this, we can apply Lemma \ref{lem: L of convol}  and proceed analogously as in the proof of  Theorem \ref{thm: formula for heat kernel} to get that

\begin{align*}
L_{x}\tilde{H}(x,y;t)&=
L_{x}H(x,y;t)+L_{x}(H\ast F)(x,y;t)
\\&=
L_{x}H(x,y;t)+\sum_{\ell=1}^{\infty}(-1)^{\ell}(L_{x}
{H})^{\ast\ell}(x,y;t)+(L_{x}{H})*\left(\sum_{\ell=1}^{\infty}(-1)^{\ell}(L_{x}{H})^{\ast\ell}\right)(x,y;t)\\
&
=0,
\end{align*}
which completes the proof.
\end{proof}


\subsection{Parametrix construction of the heat kernel in a Hilbert space}

In this section we describe the parametrix construction of the heat kernel on a Hilbert space $\mathcal{H}$ of functions $f:\X\to\R$. Throughout this section the space $\mathcal{H}$ is fixed. We denote by $\{\varphi_j\}_{j\geq 1}$ its orthonormal basis.

\begin{definition}\label{def. parametrixH} (\textbf{An $\mathcal{H}$ parametrix.})
Let $k\geq 0$ be an integer. An \emph{$\mathcal{H}$ parametrix} $H$ order $k$  for the heat operator $L$ on $(\X,\mu)$ is any continuous function $H=H(x,y;t):\X\times \X\times \mathbb{R}_{>0}$
which is smooth on $\mathbb{R}_{>0}$ in time variable $t$, belongs to $\mathcal{H}$ in each space variable, and satisfies the following properties.
\begin{enumerate}
\item The function $L_{x}H(x,y;t)$ extends to a continuous function of $t$ on  $[0,\infty)$, for all $x,y\in\X$.

 \item  For all  $x\in \X$, $t\geq0$ we have that $\Delta_{x}H(x,\cdot;t)\in\mathcal{H}$ and for all $t\in \mathbb{R}_{>0}$, we have that $\partial_t  H(x,\cdot;t)\in\mathcal{H}$.
\item Let $$a_j(x;t):= \langle L_xH(x,\cdot;t),\varphi_j(\cdot) \rangle_{\mathcal{H}},\,\, f(x,y;t):= \sum_{j\geq 1}a_j(x;t)\varphi_j(y),$$
    $x,y\in\X$, $t>0$. Then, $a_j(\cdot ;t)\in\mathcal{H}$, and $a_j(x,t)$ is continuous in $t$ variable for all $j\geq 1$, $x\in\X$ and such that $\sum_{j\geq 1}a_j(x,t)^2 =||f(x,\cdot,t)||_{\mathcal{H}}^2$ is also continuous in $t$ for all $x\in\X$. Moreover,  for all $j\geq 1$ and for any $t_0>0$ there exists a constant $C=C(t_0)$, depending only on
$t_0$, such that
$$
\sum_{j\geq 1} ||a_j(\cdot,t)||^2_{\mathcal{H}}\leq C
  \,\,\,\,\,
  \text{\rm and}
  \,\,\,\,\,
  \left\|f(\cdot, y;t)\right\|_{\mathcal{H}}^2\leq Ct^{k}.
$$
for $t \in (0, t_0]$ and all $y\in \X$.
\item For all $x\in \X$, and all $g\in \mathcal{H}$
\begin{equation}\label{eq:dirac property of heat kernel hilb}
\lim\limits_{t\rightarrow 0^+}\langle H(x,\cdot;t),g(\cdot)\rangle_{\mathcal{H}}=g(x).
\end{equation}
\end{enumerate}
\end{definition}

The construction of the heat kernel on $\Hil$ begins with the following lemma.

\begin{lemma}
\label{lem: L of convolHilb S}
Let $H$ be an $\mathcal{H}$ parametrix for the heat operator on $(\X,\mu)$ of any order.
Let $f=f(x,y;t):\X\times \X\times\mathbb{R}_{> 0}\to\mathbb{R}$
be a continuous function in $t$ for all $x,y\in \X$. Assume further that for all $t\in \mathbb{R}_{>0}$ the function
$f(\cdot,y;t)$ when viewed as a function on $\X$, for all $y\in\X$ belongs to $\mathcal{H}$.  Then
$$
L_{x}( H\ast f)(x,y;t)=f(x,y;t)+(L_{x}H \ast f)(x,y;t)
$$
for all $x,y\in \X$ and $t\in\mathbb{R}_{>0}$.
\end{lemma}

\begin{proof} From the assumptions, we have that $f(\cdot, y;t)\in \mathcal{H}$ and $H(x,\cdot;t), \, \Delta_{x}H(x,\cdot;t)\in \mathcal{H}$
for all $(x,y;t) \in \X\times \X \times \mathbb{R}_{>0}$ are continuous in variable $t$.  Therefore,
their inner product is integrable on $[0,t]$, which proves that convolutions $H\ast f$ and $(\Delta_{x}H)\ast f$ are well defined.
Next, we claim that
\begin{equation}\label{eq. delta od conv}
  \Delta_{x}(H\ast f)(x,y;t) =  \left( (\Delta_{x}H)\ast f\right)(x,y;t).
\end{equation}
In order to prove \eqref{eq. delta od conv}, let us write
$$
H(x,\cdot;t)- H(z,\cdot;t)=\sum_{j\geq 1}h_{x,j}(z;t)\varphi_j(\cdot)
\,\,\,\,
\text{\rm and}
\,\,\,\,
f(\cdot,y;t)=  \sum_{j\geq 1}g_j(y;t)\varphi_j(\cdot)
$$
Then
\begin{align} \notag
 \Delta_{x}(H\ast f)(x,y;t) &= \int_\X\left(\int_0^t \left[ \langle H(x,\cdot;t)- H(z,\cdot;t-\tau),  f(\cdot,y;\tau) \rangle_{\mathcal{H}} \right] d\tau\right)\rho_x(dz)\\ \notag
 &=\int_\X\left(\int_0^t \left[ \sum_{j\geq 1} h_{x,j}(z;t-\tau)g_j(y;\tau) \right] d\tau\right)\rho_x(dz)\\ &= \int_0^t \left( \sum_{j\geq 1} \left(\int_\X  h_{x,j}(z;t-\tau) \rho_x(dz)\right) g_j(y;\tau)  \right)d\tau\label{eq. interchange od deltax}.
\end{align}
To be precise, the above computations involved an
interchange of the sum and the integral which is justified because
\begin{align*}
\sum_{j\geq 1} \int_\X \int_0^t &| h_{x,j}(z;t-\tau)g_j(y;\tau)  | \rho_x(dz) d\tau = \int_\X \int_0^t \left(\sum_{j\geq 1}| h_{x,j}(z;t-\tau)g_j(y;\tau)  |\right)  \rho_x(dz) d\tau\\ &\leq \int_\X \int_0^t \left(\sum_{j\geq 1} h_{x,j}(z;t-\tau)^2 \right)^{1/2} \left(\sum_{j\geq 1} g_{,j}(y;\tau)^2 \right)^{1/2 }\rho_x(dz) d\tau \\&=\int_\X \int_0^t ||H(x,\cdot;t-\tau)- H(z,\cdot;t-\tau) ||_\mathcal{H} ||f(\cdot,y;\tau) ||_\mathcal{H}\rho_x(dz) d\tau <\infty,
\end{align*}
where we used that the measure $\rho_x$ is finite and all norms are bounded.

Continuing, we claim that
\begin{equation}\label{eq. intermediate}
\int_\X  h_{x,j}(z;t-\tau) \rho_x(dz) = \langle \Delta_xH(x,\cdot;t-\tau),\varphi_j(\cdot) \rangle_\mathcal{H}.
\end{equation}
Since
\begin{equation}\label{eq.DeltaH_formula}
\Delta_xH(x,\cdot;t-\tau) = \int_\X \left(\sum_{\ell\geq 1} h_{x,\ell}(z;t-\tau) \varphi_\ell (\cdot)\right) \rho_x(dz),
\end{equation}
in order to prove \eqref{eq. intermediate} it suffices to show that one can interchange the sum and the integral in \eqref{eq.DeltaH_formula}.
This easily follows from the Cauchy-Schwarz inequality, namely
$$
\sum_{\ell\geq 1} \int_\X \left| h_{x,\ell}(z;t-\tau) \varphi_\ell (\cdot) \right| \rho_x(dz)\leq \int_\X \left( \sum_{\ell\geq 1} h_{x,\ell}(z;t-\tau)^2\right)^{1/2}  \left( \sum_{\ell\geq 1} \varphi_\ell(\cdot)^2\right)^{1/2} <\infty.
$$

Now, when combining \eqref{eq. interchange od deltax} and \eqref{eq. intermediate} we arrive at
\begin{align*}
\Delta_{x}(H\ast f)(x,y;t) &= \int_0^t \left( \sum_{j\geq 1} \langle \Delta_xH(x,\cdot;t-\tau),\varphi_j(\cdot) \rangle_\mathcal{H}  g_j(y;\tau)  \right)d\tau \\&=  \int_0^t \left( \langle \Delta_x H(x,\cdot;t-\tau),f(\cdot,y;\tau) \rangle_\mathcal{H}  \right)d\tau \\&=  \left( (\Delta_{x}H)\ast f\right)(x,y;t).
\end{align*}
Therefore,
\begin{align}\notag
L_{x}(H\ast f)(x,y;t)&=
    \frac{\partial}{\partial t}(H\ast f)(x,y;t)
    + \Delta_{x}(H\ast f)(x,y;t)\\ &=
  \frac{\partial}{\partial t}(H\ast f)(x,y;t)+  \left( (\Delta_{x}H)\ast f\right)(x,y;t).\label{eq:heat_of_convolution2}
\end{align}
The function $\langle H(x,z;t-r),f(z,y;r) \rangle_{\mathcal{H}}$ is continuous in the time variable.
This follows because the integrand is continue due to the H\"older inequality and the assumptions on the parametrix, namely boundedness
uniform boundedness in $r\in[0,t]$ by an integrable function.  So we can apply the
Leibniz integration formula.  Upon doing so, we obtain that
the first term on the right hand side of \eqref{eq:heat_of_convolution2} is equal to
\begin{align} \label{eq.deriv in t}
\lim_{\varepsilon\to 0^+}\frac{\partial}{\partial t}&\int\limits _{0}^{t-\varepsilon}\left( \langle H(x,\cdot;t-\tau),f(\cdot,y;\tau) \rangle_\mathcal{H}  \right)d\tau \\&
=\lim_{\varepsilon\to 0^+}\left( \langle H(x,\cdot;\varepsilon),f(\cdot,y;\tau) \rangle_\mathcal{H}  \right) +\int\limits _{0}^{t} \frac{\partial}{\partial t}\left( \langle H(x,\cdot;t-\tau),f(\cdot,y;\tau) \rangle_\mathcal{H}  \right)d\tau .\notag
\end{align}
In view of assumption (4) on the parametrix $H$, we get from \eqref{eq.deriv in t}, combined with Lemma \ref{lem: dif of inner product wrt t} that
$$
\frac{\partial}{\partial t}\int\limits _{0}^{t}\int\limits_{\X}H(x,z;t-r)f(z,y;r) \mu(dz)dr= f(x,y;t) + \left(\frac{\partial}{\partial t}H\ast f\right)(x,y;t).
$$
Therefore,
\begin{align*}
L_{x}(H\ast f)(x,y;t) &=
  \frac{\partial}{\partial t}(H\ast f)(x,y;t)+   (\Delta_{x}H\ast f)(x,y;t)
\\&=f(x,y;t)+\left(\frac{\partial}{\partial t}H\ast f\right)(x,y;t)+(\Delta_{x}H\ast f)(x,y;t)
\\&=f(x,y;t)+(L_{x}H\ast f)(x,y;t),
\end{align*}
as claimed.
\end{proof}

\begin{theorem}
\label{thm: formula for H- heat kernel} Let $H$ be an $\mathcal{H}$ parametrix of
order $k\geq0$ for the heat operator on $\X$.  For $x,y\in \X$ and $t\in \mathbb{R}_{\geq 0}$, let
\begin{equation}\label{eq:Neuman_series H}
F(x,y;t):=\sum_{\ell=1}^{\infty}(-1)^{\ell}
(L_{x}H)^{\ast\ell}(x,y;t).
\end{equation}
Then the Neumann series \eqref{eq:Neuman_series H} converges uniformly in the Hilbert space norm $||\cdot||_{\mathcal{H}}$ on every compact subset of $\X\times \X\times\mathbb{R}_{\geq 0}$.
Furthermore, the heat kernel $K_{\X}$ on $\Hil$
associated to graph Laplacian $\Delta_{x}$ is given by \eqref{eq:heat_kernel_parametrix}
and
$$
||(H\ast F)(x,y;t)||_\mathcal{H}^2
= O(t^{k+2})
\,\,\,\,\,
\text{\rm as $t \rightarrow 0^+$.}
$$
\end{theorem}

\begin{proof}
Set
$$
\tilde{H}(x,y;t) := H(x,y;t)+(H\ast F)(x,y;t).
$$
By Lemma \ref{lemma 2,2, conv series Hilb} (with $f=L_xH$), the series $F(x,y;t)$ defined in \eqref{eq:Neuman_series H}
converges in the Hilbert space $\mathcal{H}$. Moreover, according to \eqref{eq:series_bound2 Hilb}, as a function of the first variable, it belongs to $\mathcal{H}$ and
\begin{equation}\label{eq. F in H bound}
||F(\cdot,y;t)||_\mathcal{H}^2\leq \sum_{\ell=1}^\infty \left\|(L_xH)^{\ast \ell}(\cdot,y;t)\right\|_\mathcal{H}^2=  O(t^{k}),
\end{equation}
as $t \rightarrow 0^+$. Therefore, $\tilde{H}$ satisfies assumption (1) in Definition \ref{def: HK Hilb space}. Next, we want to show that
\begin{equation}\label{eq:heat_kernel_criteria 2}
L_{x}\tilde{H}(x,y;t)=0
\,\,\,\,\,
\text{\rm and}
\,\,\,\,\,
 \lim_{t\to 0^+}\langle H (x,\cdot;t),f(\cdot)\rangle_\Hil =f(x),
\,\,\,\,
\text{\rm $x\in\X$ and $f\in \mathcal{H}$.}
\end{equation}
Assumptions on the parametrix $H$ together with \eqref{eq. F in H bound} ensure that  Lemma \ref{lem:convolution bounds Hilb} can be applied with $F_1=H$ and $F_2=F$ to deduce that
$$
\left\|H\ast F(x,y;t)\right\|_\mathcal{H}^2= O(t^{k+2})
\,\,\,\,\,
\text{\rm as $t \rightarrow 0^+$.}
$$
Therefore, for any $f\in\mathcal{H}$ and $x\in\X$
$$
 \lim_{t\to 0^+}\langle \tilde H (x,\cdot;t),f(\cdot)\rangle_\Hil =\lim_{t\to 0^+}\langle  H (x,\cdot;t)f(\cdot)\rangle =f(x).
$$
It remains to prove the vanishing of $L_{x}\tilde{H}$ in \eqref{eq:heat_kernel_criteria 2}.
For this, we can apply Lemma \ref{lem: L of convolHilb S} to get that

\begin{align*}
L_{x}\tilde{H}(x,y;t)&=
L_{x}H(x,y;t)+L_{x}(H\ast F)(x,y;t)
\\&=
L_{x}H(x,y;t)+\sum_{\ell=1}^{\infty}(-1)^{\ell}(L_{x}
{H})^{\ast\ell}(x,y;t)+(L_{x}{H})*\left(\sum_{\ell=1}^{\infty}(-1)^{\ell}(L_{x}{H})^{\ast\ell}\right)(x,y;t)\\&=
L_{x}H(x,y;t)+\sum_{\ell=1}^{\infty}(-1)^{\ell}(L_{x}
{H})^{\ast\ell}(x,y;t)+\sum_{\ell=1}^{\infty}(-1)^{\ell}(L_{x}{H})^{\ast(\ell+1)}(x,y;t)\\&
=0.
\end{align*}
This completes the proof.
\end{proof}

\section{Construction of a parametrix}

In Section \ref{sec. param. constr} assuming that the corresponding parametrix exists we constructed an $L^1$ heat kernel, an $L^2$ heat kernel and
a Hilbert space heat kernel associated to the Laplacian $\Delta$, as defined in Section \ref{sec. setup}.
Therefore, in order to construct the heat kernel in a given setting, it suffices to prove the existence of the corresponding parametrix.

As it turns out, the conditions in Defintions \ref{def. parametrix1}(4), \ref{def. parametrix2}(4) and \ref{def. parametrixH}(4) are
relatively straightforward, and the remaining
conditions for each of these defintions are not.  In the setting of Riemannian geometry, the construction of a parametrix is developed
using the Minakshishundaram-Pleijel recursive formulas.  In this section we will outline these ideas and show how it may be possible to
extend the notions in other settings.

We will leave the thorough investigation of these points for the follow-up paper \cite{JJS26}.  Nonetheless, we are presenting the discussion in this section
in order to complete the theoretical framework in this article.

\subsection{Compact Riemannian manifolds}\label{subsec:Riemannian_HK}
Let $M$ be a finite dimensional, smooth, compact Riemannian manifold of real dimension $n$.
Since $M$ is locally modeled by Euclidean space, one would expect that the
heat kernel on $M$ can be approximated  for small $t$ by the heat kernel on $n$-dimensional Euclidean space.  Specifically, and
by following \cite{MP49}, one takes a parametrix
for $y$ sufficiently close to $x$ to be given by
\begin{equation}\label{eq:Riemannian_parametrix}
H_{k}(x,y;t) = \frac{1}{(4 \pi t)^{n/2}}e^{-d^{2}_{M}(x,y)/4t}\left(\sum\limits_{j=0}^{k}u_{j}(x,y)t^{j}\right)
\end{equation}
where $k > n/2$, $d_{M}(x,y)$ is the Riemannian distance between $x$ and $y$,
and differential functions $\{u_{j}\}$.  The main objective is to have that $L_{x}H(x,y;t)$ extends to a continuous
function on $M\times M \times \mathbb{R}_{>0}$.  In order to fulfill this condition, and as discussed on page 149 of \cite{Ch84}, one
sets the \it Ansatz \rm that
\begin{equation}\label{eq:MP_ansatz}
L_{x} H_{k}(x,y,t) = t^{k}\Delta_{x}(u_{k}(x,y))\cdot \frac{1}{(4 \pi t)^{n/2}}e^{-d^{2}_{M}(x,y)/4t}.
\end{equation}
Indeed, if one can satisfy \eqref{eq:MP_ansatz}, then the first three conditions of Definitions \ref{def. parametrix1} and \ref{def. parametrix2}
will follow.  In general, the Laplacian considered in this setting satisfies the product formula that
\begin{equation}\label{eq:Laplacian_of_a_product}
\Delta (fg) = (\Delta f)g + 2\langle \text{\rm grad}f, \text{\rm grad}g \rangle + f (\Delta g),
\end{equation}
for $C^{2}$ functions on $M$ and with first derivative operator $\text{\rm grad}$; see page 3 of \cite{Ch84}.
Furthermore, as stated in \cite{Ch84}, it is advantageous to write all expressions in \eqref{eq:MP_ansatz} in geodesic spherical
coordinates near $x$ on $M$.  In doing so, one obtains from \eqref{eq:MP_ansatz}
a system of \it first-order linear differential equations \rm by equating the coefficients in $t$.  In doing so, one
gets an equation for $u_{0}$ in terms the metric on $M$, and an equation for $u_{j}$ in terms of
the metric on $M$ and $u_{j-1}$ whenever $j>0$; see pages 149-150 of \cite{Ch84}.  These equations are readily solvable,
thus leading to a somewhat explicit solution to \eqref{eq:Riemannian_parametrix}, as given on page 150 of \cite{Ch84}.

With all this, one constructs a parametrix of the form \eqref{eq:Riemannian_parametrix} from which the heat kernel is
obtained.

\begin{remark}\label{rem:non_gauss_asymptotics}
When considering the space of functions on a fractal space, it is not clear if a parametrix
similar to \eqref{eq:Riemannian_parametrix} exists. Indeed, if $K_{SG}(x,y;t)$ denotes the heat kernel on the Sierpinski gasket,
it is stated on page 70 of \cite{Ka12} that one of the main results of \cite{Ku97} is that
$$
\lim\limits_{t \rightarrow 0^{+}} t^{c}\log K_{SG}(x,y;t)
$$
does not exist for the appropriate power of $c$ and various points $x$ and $y$ on the Sierpinski gasket.  Refined results
were obtained in \cite{Ka12} using the measurable Riemannian structure; see Corollary 1.2(b).  While the author proves upper
and lower bounds for $K_{SG}(x,y;t)$ (see Corollary 1.2(a) loc. cit.), the bounds are such that one does not see a specific two-sides
Gaussian-type asymptotic which is needed in order to have a parametrix in analogy with \eqref{eq:Riemannian_parametrix}.
\end{remark}

\subsection{An $L^1$ parametrix}
Consider an $L^{1}$ space with distance measure $d(x,y)$.  Let $F : \mathbf{R}_{\geq 0} \rightarrow \mathbf{R}_{\geq 0}$ be
a smooth, non-negative  function whose Lebesgue integral is equal to one.  Without loss of generality, assume that
$F(0) \neq 0$.  It is elementary to show that the function $c(x,t)F(d(x,y)/t)$ satisfies Definition \ref{def. parametrix1}(4)
where $c(x,t)$ is determined by taking $H(x,y;t)=c(x,t)F(d(x,y)/t)$ and $f(y) \equiv 1$ in \eqref{eq:dirac property of heat kernel}; see for example page 51 of \cite{La99}.  With this, one can
set as an \it Ansatz \rm that a parametrix for the heat kernel is of the form
\begin{equation}\label{eq:L1_parametrix}
H_{k}(x,y;t) = c(x,t)F(d(x,y)/t)\left(\sum\limits_{j=0}^{k}u_{j}(x,y)t^{j}\right)
\end{equation}
together with the condition that
\begin{equation}\label{eq:MP_L1_ansatz}
L_{x} H_{k}(x,y,t) = t^{k}\Delta_{x}(u_{k}(x,y))\cdot c(x,t)F(d(x,y)/t).
\end{equation}
Implicit in the setting from Section \ref{subsec:Riemannian_HK} is that derivatives of the Euclidean heat kernel are
multiplies of the Euclidean heat kernel, thus one can factor the Euclidean heat kernel from all terms in \eqref{eq:MP_ansatz}.
Since the Euclidean heat kernel is non-vanishing,
\eqref{eq:MP_ansatz} then becomes an identity which is a polynomial in $t$ from which one obtains the Minakshishundaram-Pleijal
recursive formulas, which are solvable.

As described in Remark \ref{rem:non_gauss_asymptotics}, one cannot always assume that $F$ is Gaussian.
However, other specific examples for $F$ may be useful, at least to construct the parametrix, and thus prove the existence of the heat kernel.
For instances, one could consider $F(u) = (3/4)(1 - u^{2})$ when $\vert u \vert < 1$ and $0$ elsewhere.  In doing so, the
\it Ansatz \rm \ref{eq:MP_L1_ansatz} yields a polynomial in $t$ whose coefficients involve the $u_{j}$'s. Now, if the Laplacian
admits a structure such as \eqref{eq:Laplacian_of_a_product} and the operator $\text{\rm grad} u_{k}$ is easily inverted, as in the
geometric case when $\text{\rm grad} $ is essentially a first derivative, then one again arrives at a system of first-order linear
differential equations.  Another possibility one could consider is $F(u) = e^{-u}$.

As stated, we will leave the full development of generalized recursive formulas stemming from \eqref{eq:MP_L1_ansatz} for \cite{JJS26}.

\subsection{An $L^2$ parametrix} The discussion in the previous section applies in the $L^2$ setting as well.  Additionally,
we will point out the following perhaps trivial consideration.

Assume $L^2$ has a countable basis of eigenfunctions $\{\phi_{n}\}$ with corresponding eigenvalues $\{\lambda_{n}\}$.  As such,
one can write the parametrix $H(x,y,t)$ as
\begin{equation}\label{eq. defn H l2}
H(x,y,t) = \sum\limits_{n=0}^{\infty}a_{n,m}(t)\phi_{m}(x)\phi_{n}(y).
\end{equation}
Then, under appropriate convergence assumptions to be addressed, we have that
\begin{align*}
L_{x}H(x,y;t)= \sum\limits_{n=0}^{\infty}(L_{x}(a_{n,m}(t)\phi_{m}(x))\phi_{n}(y)
&= \sum\limits_{n=0}^{\infty}(\partial_{t}a_{n,m}(t)+\lambda_{m}a_{n,m}(t))\phi_{m}(x))\phi_{n}(y)
\\&= \sum\limits_{n=0}^{\infty}\left((\partial_{t}(e^{\lambda_{m}t}a_{n,m}(t))\right)e^{-\lambda_{m}t}\phi_{m}(x)\phi_{n}(y).
\end{align*}
If $H(x,y,t)$ is in fact the heat kernel, then $a_{n,m}(t) = 0$ if $n\neq m$ and $a_{n,n}(t) = \exp(-\lambda_{n}t)$ for all $n$.
More generally, the bounds required by  Definition \ref{def. parametrix2} are fulfilled if
$a_{n,m}(t)$  are \it close \rm to these formulas.  For example, it suffices to have that
\begin{equation}\label{eq:coeff_bounds}
\left((\partial_{t}(e^{\lambda_{m}t}a_{n,m}(t))\right) = O\left(t^{k}\lambda_{n}^{-h}\lambda_{m}^{-h}\right)
\end{equation}
for sufficiently large $k$ and $h$.
The condition \ref{def. parametrix2}(4) with $H$ defined by \eqref{eq. defn H l2}
follows from the assumption that $\lim\limits_{t \rightarrow 0}a_{n,m}(t) = \delta_{n=m}$ and \eqref{eq:coeff_bounds}.  Hence,
upon integration, we conclude that
$$
a_{n,m}(t) = \delta_{n=m}e^{-\lambda_{m}t}+O\left(t^{k}\lambda_{n}^{-h}\lambda_{m}^{-h}-\lambda_{m}t\right)
$$
for sufficiently large $h$ and $k$.  For example, one can choose $h$ large enough so that the series
$$
\sum\limits_{n=0}^{\infty}\frac{\Vert \phi_{n}\Vert_{\infty}}{\lambda_{n}^{h}} < \infty,
$$
and $k$ is picked to satisfy the conditions of Definition \ref{def. parametrix2}.

\section{Further properties of heat kernels}
In this section, under some additional assumptions we prove further properties of heat kernels, such as uniqueness, semigroup property, conservativeness and completeness.

We assume that measures $\mu$ and $\lambda$ introduced above are equal and assume that the Hilbert space $\Hil$ equals the energy space $\Hil_{\mathcal{E}}$, introduced in Section \ref{sec. setup}. Note that, when dealing with the normalized Laplacian $\tilde \Delta$, one should take $\mu=\nu$.

Then, self-adjointness property on $L^2(\lambda)$ is simply the statement that
\begin{equation}\label{eq:Green1}
\int\limits_{\X}u(x)(\Delta v)(x) \lambda(dx) = \int\limits_{\X}(\Delta u)(x)v(x) \lambda(dx),
\end{equation}
which amounts to Green's formula in many geometric settings (see e.g. Theorem 2.14. of \cite{JP19}).  In particular, we have that
\begin{equation}\label{eq:Green2}
\int\limits_{\X}(\Delta u)(x) \lambda(dx) = 0.
\end{equation}
Also, the semi-positivity of the Laplacian means that
\begin{equation}\label{eq:Green3}
\int\limits_{\X}u(x)(\Delta u)(x) \lambda(dx) \geq 0,
\end{equation}
see Lemma 2.12. of  \cite{JP19}. As always, it suffices that the above statements hold for a dense subspace of functions, such as $\mathcal{D}_{\rm fin}^\dagger$.

The self-adjointness on $\Hil$ means that
\begin{equation}\label{eq:Green1}
\langle u(x),(\Delta v)(x) \rangle_\Hil = \langle(\Delta u)(x),v(x) \rangle_\Hil,
\end{equation}
for $u,v \in \mathcal{D}_{\rm fin}^\dagger$, see e.g. \cite{JP19}.

As we will see, \eqref{eq:Green1}, \eqref{eq:Green2} and \eqref{eq:Green3}
imply that the heat kernels, which we study in the next section,
are conservative.
The discussion below follows closely
the material from \cite{Ch84}, beginning on page 136.  For the convenience of the reader,
we will repeat the most important points in the development from \cite{Ch84} for the $L^2$ setting. Proofs in the Hilbert space setting are completely analogous, so we will skip those.

\subsection{Uniqueness of the heat kernel}

We begin with the following lemma which is often referred to as \it Duhamel's principle \rm in existing literature.

\begin{lemma} \label{lem:Duhamel}
Let $u(x,t)$ and $v(x,t)$ be integrable functions on the product space $(\X \times \mathbb{R}_{>0}, \lambda\times m)$ where $m$ denotes the Lebesgue measure on $\mathbb{R}_{>0}$, and such
that the functions $u\Delta v$, $v\Delta u$, $u\partial_{t}v$ and $v\partial_{t}u$ are also integrable.
Then for any positive real numbers $\alpha$ and $\beta$ such that $0 < \alpha < \beta <\infty$, we have that
\begin{align*}
\int\limits_{\X}&\left(u(x,t-\beta)v(x,\beta) - u(x,t-\alpha)v(x,\alpha)\right)\lambda(dx)
\\&=\int\limits_{\alpha}^{\beta}\int\limits_{\X}
\left((\Delta + \partial_{t})u(x,t-\tau)v(x,\tau) - u(x,t-\tau)(\Delta + \partial_{t})v(x,\tau)\right)\lambda(dx)d\tau.
\end{align*}
\end{lemma}

\begin{proof}
Immediately, one has that
\begin{align}\label{eq:heat_operator_expansion}\notag
(\Delta + \partial_{t})&u(x,t-\tau)v(x,\tau) - u(x,t-\tau)(\Delta + \partial_{t})v(x,\tau)
\\&= (\Delta u)(x,t-\tau)v(x,\tau)  - u(x,t-\tau)(\Delta v)(x,\tau)
+ \partial_{t} (u(x,t-\tau)v(x,t)).
\end{align}
From \eqref{eq:Green1}, in view of the fact that $\mathcal{D}_{\rm fin}^\dagger$ is dense,
we have that
$$
\int\limits_{\X}\left((\Delta u)(x,t-\tau)v(x,\tau) - u(x,t-\tau)(\Delta v)(x,\tau)\right)\lambda(dx) = 0.
$$
With this, the lemma follows by integrating \eqref{eq:heat_operator_expansion} with respect to
$\tau \in \R$ over $[\alpha, \beta]$.
\end{proof}

\begin{proposition} \label{thm:uniqueness}
Assume that an $L^p(\lambda)$ heat kernel $K(x,y;t)$ (for some $p\geq 1$) is such that for arbitrary, fixed $x,y\in\X$, functions $u(\cdot,t):=K(x,\cdot, t)$ and $v(\cdot,t):=K(\cdot, y, t)$ satisfy
the conditions of Lemma \ref{lem:Duhamel}.  Then, such a heat kernel is unique and, furthermore,
$$
K(x,y;t) = K(y,x,t)
\,\,\,\,\,
\text{\rm for all $t \in \mathbb{R}_{>0}$.}
$$
\end{proposition}

\begin{proof}
The proof is the same as that of Theorem 1, page 138 of \cite{Ch84}, which, for the convenience
of the reader, we now repeat with slightly different notation.

Assume $K_{1}(x,y;t)$ and $K_{2}(x,y;t)$ be two heat kernels.  Then apply Duhamel's principle,
Lemma \ref{lem:Duhamel}, with
$$
u(z,\tau) = K_{1}(x,z;\tau)
\,\,\,\,\,
\text{\rm and}
\,\,\,\,\,
v(z,\tau) = K_{2}(y,z;\tau),
$$
which gives that
$$
\int\limits_{\X}\left(K_{1}(x,z;t-\beta)K_{2}(y,z;\beta) - K_{1}(x,z;t-\alpha)K_{2}(y,z;\alpha)\right)\lambda(dz) = 0.
$$
With these choices, let $\alpha \rightarrow 0^{+}$ and $\beta \rightarrow t^{-}$ to conclude from the Dirac property of the heat kernel that
\begin{equation}\label{eq. symmetry}
K_{2}(y,x;t) - K_{1}(x,y;t)=0.
\end{equation}
Repeating the same procedure with
$$
u(z,\tau) = K_{1}(x,z;\tau)
\,\,\,\,\,
\text{\rm and}
\,\,\,\,\,
v(z,\tau) = K_{1}(y,z;\tau),
$$
we deduce that
$$
K_{1}(y,x;t) - K_{1}(x,y;t)=0
$$
Hence, we conclude that $K_{1}(x,y;t)=K_{1}(y,x;t)$, or that $K_{1}$ is symmetric in the space variables.  Second, having showed symmetry, from \eqref{eq. symmetry} we now conclude that $K_{1}(x,y;t) = K_{2}(x,y;t)$, which is uniqueness.
\end{proof}

\subsection{Dependence on the initial condition} Let us now discuss the assumptions under which solutions to the heat equation are uniquely determined by their limiting values as $t\rightarrow 0^{+}$.

\begin{lemma} \label{lem:time_dependence}
Let $u(x,t)$ be a solution of the heat equation $(\Delta+\partial_t)u=0$ and such that $u$, $\Delta u$ and $u\Delta u$ are in $L^1(\lambda)$.  Then
the integral
$$
\int\limits_{\X}u(x,t)\lambda(dx)
$$
is constant in $t$, and the integral
$$
\int\limits_{\X}u^{2}(x,t)\lambda(dx)
$$
is decreasing in $t$.
\end{lemma}

\begin{proof}
For the first assertion, simply note that
$$
\partial_{t}\int\limits_{\X}u(x,t)\lambda(dx) = \int\limits_{\X}\partial_{t}u(x,t)\lambda(dx)
= \int\limits_{\X}(-\Delta) u(x,t)\lambda(dx) =0.
$$
The proof of the second assertion is similar since
\begin{align*}
\partial_{t}\int\limits_{\X}u^2(x,t)\lambda(dx) &= \int\limits_{\X}\partial_{t}u^2(x,t)\lambda(dx)
= 2\int\limits_{\X}u(x,t)\partial_{t}u(x,t)\lambda(dx)
\\&= -2\int\limits_{\X}u(x,t)\Delta u(x,t)\lambda(dx) \leq 0,
\end{align*}
where the last inequality follows from the semi-positivity of the Laplacian.
\end{proof}

\begin{remark} The above integrability conditions imply that solutions to the heat
equation are uniquely determined by their limiting values as $t\rightarrow 0^{+}$ under
these conditions:  Assume that $u(x,t)$ and $v(x,t)$ are solutions to the heat equation satisfying assumptions of Lemma \ref{lem:time_dependence} and such that
$$
\lim\limits_{t\rightarrow 0^{+}} u(x,t) = \lim\limits_{t\rightarrow 0^{+}} v(x,t).
$$
Assume further that
$u\Delta v$, $v\Delta u$ and $uv$ are integrable.  Then, we conclude
that $u(x,t) = v(x,t)$ for all $x\in \X$ and $t\in \mathbb{R}_{>0}$.  Indeed, the difference
$D(x,t):= u(x,t) - v(x,t)$ is such that
$$
\int\limits_{\X} D(x,t)\lambda(dx) = 0
\,\,\,\,\,
\text{\rm for all $t \in \mathbb{R}_{>0}$,}
$$
because
$$
\lim\limits_{t \rightarrow 0^+}\int\limits_{\X} D(x,t)\lambda(dx) = 0.
$$
Furthermore,
$$
\int\limits_{\X} D^{2}(x,t)\lambda(dx)= 0
\,\,\,\,\,
\text{\rm for all $t \in \mathbb{R}_{>0}$,}
$$
because the integral is non-negative, zero when $t\rightarrow 0^{+}$, and decreasing.  Therefore,
$D(x,t)$ is zero for all $x$ and $t$.
\end{remark}

\begin{remark}
It is at this point we can highlight the famous example due to Tychnoff which shows that the assumptions of integrability given above are necessary condition for the statement of Lemma \ref{lem:time_dependence} to hold true.  Define the function
$\phi(t) = e^{-1/t^{2}}$ for $t > 0$, and $\phi(t) = 0$ for $t < 0$.  Then the function
\begin{equation}\label{eq:Tychnoff}
u(x;t):= \sum\limits_{k=0}^{\infty} \frac{\phi^{(2k)}(t)}{(2k)!}x^{2k}
\end{equation}
satisfies the heat equation $\left(\frac{\partial^2}{\partial x^2} + \frac{\partial}{\partial t}\right)u(x;t)=0$ with initial condition
$$
\lim\limits_{t \rightarrow 0^{+}} u(x;t) = 0
\,\,\,\,\,
\text{\rm for all $x \in \mathbb{R}$.}
$$
The heat kernel on $\mathbb{R}$ is well-known, and it is
$$
K_{\mathbb{R}}(x,y;t) = \frac{1}{\sqrt{4\pi t}}e^{-(x-y)^{2}/4t}.
$$
The function \eqref{eq:Tychnoff} falls outside the range of the restrictions stated above because
$u(x,t)$ because it can be shown that $K_{\mathbb{R}}(x,y,t_{1})u(x,t_{2})$ is not integrable in $x$
in $y$ for any positive $t_{1}$ and $t_{2}$.
\end{remark}

\subsection{Semi-group property}

\begin{proposition} \label{prop:hk_properties}
Assume that the heat kernel $K(x,y,t)$ satisfies the properties of Definition \ref{def: HK lp space}.  Also,
assume that $u(x,t):=K(x,\cdot, t)$ and $v(y,t):=K(\cdot,y,t)$ satisfy the conditions of Lemma \ref{lem:Duhamel} and
Lemma \ref{lem:time_dependence}.  Then the heat kernel satisfies the semi-group property that
\begin{equation}\label{eq:semigroup}
K(x,y,t+s) = \int\limits_{\X}K(x,z,t)K(z,y,s)\lambda(dz)
\end{equation}
for all $x,y \in \X$ and $t, s >0$.  For any $x\in\X$ and $t > 0$,
\begin{equation}\label{eq:stochastically_bounded}
\int\limits_{\X}K(x,y;t)\lambda(dy) \leq 1.
\end{equation}
For all $x,y \in \X$ and $t \in \mathbb{R}_{>0}$,
\begin{equation}\label{eq:positive}
K(x,y;t) \geq 0
\end{equation}
\end{proposition}

\begin{proof}
For the sake of space, we will outline the main steps of proof and refer to
page 139 of \cite{Ch84} for further details.

For any function $f(x)\in \D_{\rm fin}$ consider the integrals
\begin{equation}\label{eq:semigroup_proof}
\int\limits_{\X}K(x,y,t+s)f(y)\lambda(dy)
\,\,\,\,\,
\text{\rm and}
\,\,\,\,\,
\int\limits_{\X}\left(\int\limits_{\X}K(x,z,t)K(z,y,s)\lambda(dz)\right)f(y)\lambda(dy).
\end{equation}
Both integrals are solutions to the heat equation with initial condition, in the variable $t$,
given by
$$
\int\limits_{\X}K(x,y;t)f(y)\lambda(dy).
$$
Therefore, the two integrals in \eqref{eq:semigroup_proof} are equal.  Since $f$ is arbitrary, this
implies that the kernel functions are equal, which is precisely the assertion \eqref{eq:semigroup}.

Let $f(x) = \chi_{A}$ for any $A\in \mathcal{B}_{\rm fin}$, we get from Lemma \ref{lem:time_dependence}
that
$$
\int\limits_{A}K(x,y;t)\lambda(dy) \leq 1.
$$
By considering a sequence of sets $A$ which exhaust $\X$, the assertion \eqref{eq:stochastically_bounded} follows.

For any function $f(x)$, we have from \eqref{eq:semigroup} that
\begin{align*}
\int\limits_{\X}\int\limits_{\X}f(x)K(x,y;t)f(y)\lambda(dx)\lambda(dy)
=
 \left(\int\limits_{\X}\int\limits_{\X}f(y)K(z,y,t/2)\lambda(dy)\lambda(dz)\right)^{2} \geq 0.
\end{align*}
Since $f$ is arbitrary, we arrive at the assertion \eqref{eq:positive}.
\end{proof}

\begin{remark} The assertions in Proposition \ref{prop:hk_properties} combine to
give that the operator defined by
\begin{equation}\label{eq:positive_operator}
u \mapsto \int\limits_{\X}K(\cdot,y,t)u(y)\lambda(dy)
\end{equation}
is positive and has norm bounded by one.
\end{remark}

\begin{remark}
If
$$
\int\limits_{\X}K(x,y;t)\lambda(dy) = 1.
$$
the one refers to the heat kernel as \it stochastically complete. \rm
It is an entirely different undertaking to determine the settings, even in the presence of
geometry, when a heat kernel is stochastically complete.
\end{remark}

\begin{remark}
As stated in the introduction, Grigor'yan in \cite{Gr03} gives an extremely well-written
definition and subsequent discussion of heat kernels on measure spaces which defines a number
of properties for an object to be called a heat kernel (see Definition 2.1, page 146 of \cite{Gr03})
after which an associated operator of Laplacian type has been determined.  For us, we begin
with the operator and the main property that the heat kernel is an approximation of the identity,
from which we prove the existence of the heat kernel through a parametrix construction.  The above
properties establish additional properties of the heat kernel, the collection of which addresses
all of the other properties assumed by \cite{Gr03}.  As such, we view our analysis as complementary
to \cite{Gr03}.
\end{remark}

\section{Green's function, resistance measures, and other derived functions}

In this section we discuss how one can derive other well-known quantities from the heat kernel.  Most of
the material in the section points toward future areas of investigation.

\subsection{Green's function}

The Green's function, also known as the resolvent kernel is the kernel of the inverse of the Laplacian, if such an inverse exists. The inverse exists if and only if zero is not an eigenvalue of the Laplacian. In the setting when the Laplacian is probabilistic and the time is discrete, the heat kernel can be viewed as a Markov process and in that case existence of the inverse can be viewed as transiency property of the random walk. Its analogue in the continuous time case is the assumption that the time-integral of the heat kernel is finite.
Actually, the existence of the inverse follows from the following two assumptions we pose on the heat kernel.

Assume that the space $(\X,\mu)$ is such that there exists a heat kernel $K_\X$ on this space and that it belongs to $L^p(\mu)$ ($p\geq1$).
In addition,  we assume the following two properties of $K_{\X}$.

\vskip .10in
\begin{enumerate}
  \item [HK1] For all $x,y\in\X$, the functions $t\mapsto K_\X(x,y;t)$ and $t\mapsto\partial_t K_\X(x,y;t)$ are  integrable on $\mathbb{R}_{\geq 0}$.
  \item [HK2] For all $x,y\in\X$
  $$
  \int\limits_\X\int\limits_0^\infty|K_\X(x,z;t) - H_\X(y,z;t)|dt\rho_x(dy)<\infty.
  $$
\end{enumerate}

\vskip .10in
In view of the fact that
$$
\int\limits_\X\int\limits_0^\infty|H_\X(x,z;t)|dt\rho_x(dy)= c(x)\int\limits_0^\infty|H_\X(x,z;t)|dt
$$
the assumption HK2 can be replaced with
$$
 \int\limits_\X\int\limits_0^\infty| H_\X(y,z;t)|dt\rho_x(dy)<\infty.
$$
With those two assumptions we have the following proposition.

\begin{proposition}
  With the notation as above, assume that the heat kernel $K_{\X}\in L^p(\mu)$ for $p=1$ or $p=2$ satisfies assumptions HK1 and HK2. Then,
  $$
  G(x,y):=\int\limits_0^\infty K_\X(x,y;t)dt
  $$
  is the kernel of the inverse of the Laplacian, meaning that for every $f\in L^q(\mu)$ with $1/p+1/q=1$ we have that
  $$
  \Delta\left(\int\limits_\X G(x,y) f(y) \mu(dy)\right)=f(x).
  $$
\end{proposition}
\begin{proof}
  From the definition of the Laplacian we have
  \begin{align*}
  \Delta\left(\int\limits_\X G(x,y) f(y) \mu(d y)\right) &= \int\limits_\X\left(\int\limits_\X G(x,z) f(z) \mu(dz) - \int\limits_\X G(y,z) f(z) \mu(dz)\right)\rho_x(dy)\\
  &= \int\limits_\X  f(z) \int\limits_\X (G(x,z) - G(y,z))\rho_x(dy) \mu(dz)\\
  &=  \int\limits_\X  f(z) \int\limits_\X  \int\limits_0^\infty (K_\X(x,z;t) - K_\X(y,z;t))dt\rho_x(dy) \mu(dz).
  \end{align*}
  Using HK2 and by applying Fubini-Tonelli theorem we get that
  \begin{align*}
  \Delta\left(\int\limits_\X G(x,y) f(y) \mu(dy)\right) &=  \int\limits_\X  f(z) \int\limits_0^\infty \left( \int\limits_\X  (K_\X(x,z;t) - K_\X(y,z;t)) \rho_x(dy)\right) dt \mu(dz)\\
  &=\int\limits_\X  f(z) \int\limits_0^\infty \Delta K_\X(x,y;t) dt \mu(dz)
 \\ &= \int\limits_\X  f(z) \int\limits_0^\infty \Delta K_\X(x,y;t) dt \mu(dz)\\
 &=  - \lim\limits_{\varepsilon \rightarrow 0}\int\limits_\X  f(z) \int\limits_{\varepsilon}^\infty \partial_t K_\X(x,y;t) dt \mu(dz) \\
 &=  \lim\limits_{\varepsilon \rightarrow 0}\int\limits_\X K_\X(x,y;\varepsilon) f(z)\mu(dz) =f(x),
 \end{align*}
which completes the proof.
\end{proof}

\begin{remark}\label{rem:inverse_green}
In many instances where $(\X,\mu)$ is a Hilbert space the Laplacian $\Delta$ does, in fact, have zero eigenvalues.  In that
case, one simply considers the subspace which is the orthogonal complement of the space of functions spanned by solutions
to $\Delta f = 0$ and proceeds as above.  Equivalently, one can consider the integral
$$
G(x,y;s):=\int\limits_0^\infty e^{-st}K_\X(x,y;t)dt,
$$
which yields an inverse to the operator $\Delta - s$, using the proof given above.  Then the Green's function is
gotten by proving that $G(x,y;s)$ admits a meromorphic continuation to $s=0$ with a simple pole, and the Green's function
is the constant term in the Laurent expansion at $s=0$.
\end{remark}

\subsection{Resistance measure}

In section 3.2 of\cite{JP23} the authors studied different types of resistance metrics on infinite networks.
We refer the reader to \cite{JP23} for a discussion of the key ideas for the two main metrics, namely
 the free resistance $R^F$ and the wired resistance $R^W$.  On a finite graph, the resistance metric is a function
 of two points $x$ and $y$ and it is defined to be $R(x,y): = v(x)-v(y)$ where $v$ is the unique solution to $\Delta v = \delta_{x} -\delta_{y}$ where $\delta_x$ is a unit Dirac mass at $x$;
 see page 54 of \cite{JP23}.  It is later shown in \cite{JP23} that there are different ways to extend the notion of a resistance metric
 to an infinite graph, thus yielding the two concepts $R^{F}$ (Definition 3.8, page 59) and $R^{W}$ (Definition 3.16, page 63).
Let us now present a third approach.

In \cite{Wu04}, the resistance between two nodes of a finite graph is defined in terms of entries of the regularized Green's matrix.
Namely, on a finite graph, the combinatorial Laplacian $\Delta$ has the eigenvalue zero with multiplicity one;  hence the inverse is not
well defined. One can take the Moore-Penrose inverse or regularize the inverse as in Remark \ref{rem:inverse_green}.
Let $G^\ast$ denote the regularized Green's function.  Then, in the setting of a finite graph, one can show that

\vskip .10in
\begin{equation}\label{eq. resistance defn}
R(x,y)=G^\ast(x,x)+G^\ast(y,y)-G^\ast(x,y)-G^\ast(y,x)
\end{equation}
is a metric.

\vskip .10in
\noindent
It seems natural to consider \eqref{eq. resistance defn} as the definition of a resistance metric in the general
setting considered in the present article, or perhaps in the general setting of nonatomic networks, as in section 4 of \cite{JP19}.
At this time, we have not determined what conditions imply that $R(x,y)$ is a metric, or how \eqref{eq. resistance defn}
is related to $R^{F}$ and $R^{W}$ in the setting studied in \cite{JP23}.  We will leave this study for a future article.

\subsection{Entropy}
Assume that $(\X,\mu)$ is such that the associated heat kernel is stochastically complete, meaning that
$$
\int\limits_{\X}K(x,y;t)\mu(dx) = 1.
$$
Then for fixed $x$ one can consider the associated entropy, which is defined by
$$
E(x,t):=\int\limits_{\X}K(x,y;t)\log K(x,y;t) \mu(dy).
$$
In some articles, $E(x,t)$ is, up to sign, called the Shannon entropy.  In \cite{MPR24}, the quantity $E(x,t)$
is called the Kullback-Leibler divergence in the case when $(\X,\mu)$ is the $L^{2}$ space of square integrable
functions on a compact Riemannian manifold associated to a normalized Riemannian volume.  In \cite{MPR24}
the authors study the asymptotic expansion of $E(x,t)$ as $t \rightarrow 0+$, and on page 2 state the following
as motivation of their work, which we quote.

\vskip .10in \noindent
\it
Our mathematical result, or at least the concrete asymptotic expansion of the entropy up to some fixed order, is needed in a
machine-learning algorithm called the Diffusion Variational Autoencoder (VAE), which is a variant of a Variational Autoencoder
(VAE) that allows for a closed manifold as latent space, rather than the usual Euclidean latent space.
\rm

\vskip .10in
This is intriguing assertion, as is the fact that the main ingredient in the analysis of \cite{MPR24} is the parametrix construction
of the heat kernel on a compact Riemannian manifold; see paragraph 2, page 2 of \cite{MPR24}.  As such, the results in
above are poised to undertake the problem of generalizing the results of \cite{MPR24} to the setting considered here.

\subsection{Poisson and wave kernels}
In somewhat vague terms, the Green's function is a fundamental solution to $\Delta f(x) = 0$, and the heat kernel
is a fundamental solution to $(\Delta + \partial_{t})u(x,t) = 0$.  The operators $\Delta$ and $\Delta + \partial_{t}$
are elliptic and parabolic, so in a sense there remains to consider the hyperbolic operator $\Delta + \partial_{t}^{2}$,
which is commonly called the wave operator.

A classical integral formula states that
$$
e^{-wa} = \frac{w}{\sqrt{4\pi}}\int\limits_{0}^{\infty}e^{-ta^2}e^{-\frac{w^{2}}{4t}}t^{-1/2}\frac{dt}{t}.
$$
Formally, if one writes the heat kernel as $e^{-t\Delta}$ and the wave kernel as $e^{-w\sqrt{\Delta}}$, then
one would expect that
\begin{equation}\label{eq:heat_to_wave}
e^{-w\sqrt{\Delta}} = \frac{w}{\sqrt{4\pi}}\int\limits_{0}^{\infty}e^{-t\Delta}e^{-\frac{w^{2}}{4t}}t^{-1/2}\frac{dt}{t}.
\end{equation}
We refer to \cite{JL03} which discusses, albeit briefly, some aspects of the integral transformation \eqref{eq:heat_to_wave}
known as subordination and its place in probability theory going back to L\'evy \cite{Le39}.  Specfically, if
$w$ is a real variable then one obtains the Poisson kernel and if $w$ is purely complex then one obtains the wave kernel.

In \cite{JL03} it is shown that one can use the small time asymptotic behavior of a heat kernel to prove that the wave
kernel admits a branched meromorphic continuation in time with singularities in (complex) time equal to the distance
between the space variables $x$ and $y$.  In full generality, this result is known as the Duistermaat-Guillemin theorem.
If one considers heat kernels associated to Laplacians on finite volume hyperbolic Riemann surfaces, then it is shown
in \cite{JvPS16} that one can apply this methodology to deduce an integral representation for certain Eisenstein series, including the non-holomorphic
parabolic Eisenstein series.  It would be interesting to develop this line of thought on infinite graphs, for example,
to determine if one can obtain a precise expression for the spectral measure as one has in the setting of non-compact
finite volume hyperbolic Riemann surfaces.

\subsection{Heat kernel identities}\label{sec:heat_kernel_uniqueness}

Under certain circumstances, the heat kernel is unique, so then any two expressions yield an identity.
One such expression in the setting of a separable Hilbert space is (formally) of the form
$$
K_{\X}(x,y;t) = \sum e^{-\lambda_{n}t}\phi_{n}(x)\phi_{n}(y),
$$
where $\{\phi_{n}\}$ is a complete orthonormal basis of square-integrable eigenfunctions with
eigenvalue $\lambda_{n}$.  As is well-known, by comparing the spectral expansion for the heat kernel
and the parametrix construction, one obtains Weyl-type formulas for the number of eigenvalues bounded
by $T$, as an asymptotic formula in $T$, as well as the asymptotic growth of the values $\{\phi_{n}^{2}(x)\}$.
The generalization of these results to the setting of this article now are readily available.

\section{Further examples and open questions}

Let us conclude this article by listing a number of examples.  Within many of the examples
we will state points which we believe are interesting directions for future studies.

\subsection{Finite and infinite graphs}
There are many various examples to consider here, so we will list only some of the possibilities.

Suppose $\X$ is a finite graph, so $\mathcal{B}$ is generated by a finite subsets of singleton sets,
each with positive measure.  Then, as discussed in \cite{CJKS24} and \cite{JKS26}, different examples
yield various interesting results.  For example, Proposition 19 of \cite{CJKS24} studies a discrete
half-line and derives the heat kernel (with a boundary condition) using the heat kernel on the discrete line.
Along the way, the proof implies a number of identities associated to the classical
$I$-Bessel function.  In \cite{JKS26} it is shown that one can take a Dirac
delta function as a parametrix, which yields a known combinatorial formula for the heat kernel on an infinite
graph in terms of lengths of certain paths.

\subsection{Reproducing kernel Hilbert space}
Consider the setting of a reproducing kernel Hilbert space with kernel $H(x,y)$.  Then
the function $e^{-t}H(x,y)$ is a parametrix.

In general, the function $H(x,y)$ is a
called the Bergman kernel, and itself is a widely studied function and it possesses
``existence and uniqueness'' features like the heat kernel, under certain circumstances.
Indeed, this is the point of view from Chapter 6 of \cite{Mi06} who is following a method
due to Selberg to get (what is called) the Eichler trace formula by constructing the Bergman kernel
using a second approach.    (For a review of this material using
the language of reproducing kernel Hilbert spaces, we refer to unpublished yet available
online notes by E. Assing, who refers to the Master's Thesis of F. V\"olz.)

In this setting, the small time asymptotics of the heat kernel was used to prove effective sup-norm
bounds for the Bergman kernel associated to the space of holomorphic form of fixed
weight for finite volume Riemann surfaces; see \cite{FJK19}.  The same method of
proof has been employed in other settings; see for example \cite{ABR25}, \cite{AKvP25} and references therein.
It is evident in all of these cases that the proof involves general properties of heat kernels,
and is widely adaptable under general circumstances.

There is a very general construction of positive definite kernels, which has many applications;
see section 6 of \cite{JT21}. Namely,  let $\{f_{n}\}$ be an orthonormal basis of the separable Hilbert space $\X$
under study, finite or not, and let $\{Z_{n}\}$ be an i.i.d. system of standard Gaussians.  For every
$x \in \X$, set $G_x = \sum f_n(x) Z_n$.  Then the positive definite kernel $K(x,y) = \sum f_{n}(x) f_{n}(y)$
can be realized as $K(x,y) = \mathbb{E}(G_x G_y)$, where $\mathbb{E}$ denotes the expected value.
We refer to \cite{JT21}, specifically Theorem 6.3, and references therein for further discussion and
applications.  The results in the present article now provide an explicit construction of the
corresponding heat kernel, thus makes available all the corresponding mathematical tools for future studies.

\subsection{Non-separable Hilbert spaces} In the setting of Riemannian geometry, one can define
a heat kernel on a non-compact space through compact exhaustion; see, for example, Chapter VIII of \cite{Ch84},
specifically page 188.  If $M$ is a noncompact Riemannian manifold, one considers a sequence
$\Omega_{j}$ for integers $j \geq 1$ whose compactification $\bar{\Omega_{j}}$ lies in $\Omega_{j+1}$
and such that $\displaystyle \bigcup\limits_{j=1}^{\infty} \Omega_{j} = M$.  Then one defines the
heat kernel on $M$ to be the limit in $j$ of the heat kernels on $\Omega_{j}$.  We hypothesize that similarly
one can define a heat kernel on a nonseparable Hilbert space through compact exhaustion. Again, we leave that
study for a future investigation.




\subsection{Changing metrics}

There are instances where a heat kernel in one setting can serve as a parametrix for another.
For example, consider the setting when $\X$ consists of the vertices of a finite or infinite graph.
Assume one has a heat kernel $K_{\lambda}(x,y;t)$ associated to one measure $\lambda$ whose associated
degree function is uniformly bounded away from zero and infinity.  In this case, the small time
behavior of the heat kernel is simply $K_{\lambda}(x,y,0) = \delta_{x=y}$, meaning that the limiting
behavior of $K_{\lambda}(x,y;t)$ as $t \rightarrow 0^{+}$ is zero if $x \neq y$ and $1$ if $x=y$.
Let $\lambda'$ be another such measure.  Then it is straightforward to show that $K_{\lambda}(x,y;t)$
is a parametrix for the heat kernel $K_{\lambda'}(x,y;t)$.

As similar argument holds in the setting of reproducing kernel Hilbert spaces.

\section*{Acknowledgements}
The second author acknowledges grant support from PSC-CUNY Award 67415-00-55, which was jointly funded by the
Professional Staff Congress
and The City University of New York.  The second and third authors thank Gautam Chinta and Anders Karlsson for their many
fruitful mathematical discussions.

\thispagestyle{empty}
{\footnotesize
\bibliographystyle{amsalpha}
\bibliography{reference_JJS1}

\vspace{5mm}\noindent
Palle Jorgensen \\
Department of Mathematics \\
University of Iowa \\
25B MacLean Hall (MLH) \\
2 W. Washington Street \\
Iowa City, IA 52240 \\
U.S.A. \\
e-mail: palle-jorgensen@uiowa.edu

\vspace{5mm}\noindent
Jay Jorgenson \\
Department of Mathematics \\
The City College of New York \\
Convent Avenue at 138th Street \\
New York, NY 10031
U.S.A. \\
e-mail: jjorgenson@mindspring.com

\vspace{5mm}\noindent
Lejla Smajlovi\'c \\
Department of Mathematics \\
University of Sarajevo\\
Zmaja od Bosne 35, 71 000 Sarajevo\\
Bosnia and Herzegovina\\
e-mail: lejlas@pmf.unsa.ba

\end{document}